\documentclass[a4paper,12pt]{article}

\usepackage{amsmath,amssymb,amsthm,amsfonts,graphicx,mathrsfs,cancel,tikz,genyoungtabtikz,enumitem,hyperref,mathtools,nccmath,needspace}
\usepackage[a4paper,margin=1in]{geometry}

\DeclareMathOperator{\std}{Std}
\DeclareMathOperator{\End}{End}
\DeclareMathOperator{\Char}{char}
\DeclareMathOperator{\Dom}{Dom}
\newcommand{\Psichaind}[2]{\Psi\hspace{-4pt}\underset{\scriptscriptstyle{#2}}{\overset{\scriptscriptstyle{#1}}{\downarrow}}}
\newcommand{\Psichainu}[2]{\Psi\hspace{-4pt}\underset{\scriptscriptstyle{#1}}{\overset{\scriptscriptstyle{#2}}{\uparrow}}}
\newcommand{\ten}{10}
\newcommand{\eleven}{11}
\newcommand{\twelve}{12}
\newcommand{\thirteen}{13}
\allowdisplaybreaks
\def\Item{\item\abovedisplayskip=0pt\abovedisplayshortskip=0pt~\vspace*{-\baselineskip}}

\date{}
\title{Decomposable Specht modules for the Iwahori--Hecke algebra $\mathscr{H}_{\mathbb{F},-1}(\mathfrak{S}_n)$}
\author{Liron Speyer\\\normalsize Queen Mary University of London, Mile End Road, London E1 4NS, UK\\\texttt{\normalsize l.speyer@qmul.ac.uk}}

\newtheorem{lem}{Lemma}[section]
\newtheorem{thm}[lem]{Theorem}
\newtheorem{cor}[lem]{Corollary}
\newtheorem{prop}[lem]{Proposition}
\theoremstyle{definition}
\newtheorem{defn}[lem]{Definition}
\newtheorem*{eg}{Example}
\theoremstyle{remark}
\newtheorem*{ack}{Acknowledgements}
\newtheorem*{rem}{Remark}
\newtheorem*{notn}{Notation}

\begin{document}
\maketitle
\begin{center}
2000 Mathematics subject classification: 20C08, 20C30
\end{center}
\begin{abstract}
Let $S_\lambda$ denote the Specht module defined by Dipper and James for the Iwahori--Hecke algebra $\mathscr{H}_n$ of the symmetric group $\mathfrak{S}_n$. When $e=2$ we determine the decomposability of all Specht modules corresponding to hook partitions $(a,1^b)$. We do so by utilising the Brundan--Kleshchev isomorphism between $\mathscr{H}$ and a Khovanov--Lauda--Rouquier algebra and working with the relevant KLR algebra, using the set-up of Kleshchev--Mathas--Ram. When $n$ is even, we easily arrive at the conclusion that $S_\lambda$ is indecomposable. When $n$ is odd, we find an endomorphism of $S_\lambda$ and use it to obtain a generalised eigenspace decomposition of $S_\lambda$.
\end{abstract}

\section{Introduction}
Let $\mathfrak{S}_n$ denote the symmetric group on $n$ letters. The representation theory of $\mathfrak{S}_n$ largely centres around studying a certain family of $\mathbb{F}\mathfrak{S}_n$-modules $S_\lambda$, called Specht modules, which are indexed by partitions $\lambda$ of $n$. When the characteristic of $\mathbb{F}$ is zero, this family gives a complete set of pairwise non-isomorphic $\mathbb{F}\mathfrak{S}_n$-modules. The modern standpoint on the theory of Specht modules has been developed by James and can be found in \cite{j}.

The Iwahori--Hecke algebra $\mathscr{H}_{\mathbb{F},q}(\mathfrak{S}_n)$ of $\mathfrak{S}_n$ is a deformation of the group algebra $\mathbb{F}\mathfrak{S}_n$. One motivation for its study is its representation theory -- it provides a bridge between the representation theory of the symmetric and general linear groups. Here we consider it interesting in its own right.

In the series of papers \cite{dj1}, \cite{dj2} and \cite{dj3}, Dipper and James laid the foundations for and built up the theory of the Iwahori--Hecke algebra from the ground to a point where much was understood about its representation theory, at least when the quantum characteristic (denoted by $e$) is not $2$. $\mathscr{H}_{\mathbb{F},q}$ has a theory of Specht modules analogous to that for $\mathbb{F}\mathfrak{S}_n$, which describes the classical situation as a special case. Combinatorics of partitions and tableaux play an important role here, as in the classical case. Interestingly, results that hold for $\mathbb{F}\mathfrak{S}_n$ in characteristic $p$ can often be restated for $\mathscr{H}_{\mathbb{F},q}$ at quantum characteristic $p$. An example of this is the fact that Specht modules for $\mathbb{F}\mathfrak{S}_n$ are indecomposable when $\Char\mathbb{F}\neq2$ and Specht modules for $\mathscr{H}_{\mathbb{F},q}$ are indecomposable when $e\neq2$ (that is, $q\neq-1$).

When $e=2$, we would like to fill in some gaps; in particular we would like to know which Specht modules are decomposable.

In \cite{bk1}, Brundan and Kleshchev showed that, remarkably, $\mathscr{H}_{\mathbb{F},q}$ is isomorphic to a certain Khovanov--Lauda--Rouquier algebra. Immediately, this leads to a non-trivial $\mathbb{Z}$-grading on $\mathscr{H}_{\mathbb{F},q}$ and therefore $\mathbb{F}\mathfrak{S}_n$. Working from the KLR perspective, a whole theory of graded Specht modules and graded homomorphisms can be built up. One recent development has been made in \cite{kmr}, where (in an even more general setting) the authors give a presentation of graded Specht modules by generators and relations. They construct a homogeneous basis for $S_\lambda$ indexed by the standard $\lambda$-tableaux.

In this paper, we will use the KLR machinery to work with $\mathscr{H}_{\mathbb{F},q}$. We will determine the decomposability of Specht modules for hook partitions at $e=2$, building on Murphy's work on the special case $q=1$ in \cite{gm}.

In section 2, we will outline the necessary preliminaries and state some established results which we will call upon.

In section 3, we look at Specht modules for hook partitions. We begin some preliminary work towards constructing endomorphisms of these Specht modules.

In section 4, we use the results of section 3 to recover a result regarding the endomorphism algebras of Specht modules for hook partitions when $n$ is even, and thus determine their indecomposability.

In the crucial section 5, we examine the actions of the KLR generators on a special subset of the homogeneous basis for Specht modules when $n$ is odd.

Finally, in section 6, we construct an endomorphism $f$ of Specht modules $S_\lambda$, for $n$ odd. We calculate some eigenvalues of $f$ and thus determine when $S_\lambda$ can be decomposed into a direct sum of non-trivial generalised eigenspaces. In the case $\lambda=(a,1^2)$, it turns out some extra work is needed to arrive at our final result. This is achieved reasonably easily in light of the work done to establish KLR generator actions on the homogeneous basis.\\

\begin{ack}
This paper was written under the supervision of Matthew Fayers at Queen Mary University of London. The author would like to thank Dr Fayers for his helpful comments and advice throughout the project. The author must also thank Queen Mary University of London for their financial support, which has enabled this work to take place.
\end{ack}

\section{KLR algebras and their Specht modules}
\subsection{The Iwahori--Hecke Algebra}
\begin{defn}
We define the \emph{Iwahori--Hecke algebra} $\mathscr{H}=\mathscr{H}_{\mathbb{F},q} (\mathfrak{S}_n)$ of the symmetric group $\mathfrak{S}_n$ to be the unital associative $\mathbb{F}$-algebra with presentation

\[\left\langle T_1, \dots, T_{n-1} \\ \left| \begin{matrix}
(T_i-q)(T_i+1)=0\text{ for }i=1,2,\dots,n-1\\
T_iT_j=T_jT_i\text{ for }1\leqslant i<j-1\leqslant n-2\\
T_iT_{i+1}T_i=T_{i+1}T_iT_{i+1}\text{ for }i=1,2,\dots,n-2
\end{matrix}\right.\right\rangle\]
\end{defn}

\begin{defn}
Define $e\in\{2,3,4,\dots\}$ to be the smallest integer such that $1+q+q^2+\dots+q^{e-1}=0$. If no such integer exists, we define $e=0$. We call $e$ the \emph{quantum characteristic} of $\mathscr{H}$.
\end{defn}

When $e=2$ (that is $q=-1$), $\mathscr{H}$ is isomorphic by \cite{bk1} to a Khovanov--Lauda--Rouquier (KLR) algebra with the following presentation:

\vspace{12pt}

\emph{\textbf{Generators}} \quad $y_1, \dots, y_n, \qquad \psi_1, \dots, \psi_{n-1}, \qquad e(i) \quad  \text{for all } i \in \{0,1\}^n$

\vspace{12pt}

\emph{\textbf{Relations}} \quad \begin{align*} y_1 &= 0\\
e(i) &= 0 &\text{if } i_1 = 1\\
e(i) e(j) &= \delta_{ij} e(i)\\
\sum_i e(i) &= 1\\
y_r e(i) &= e(i) y_r\\
\psi_r e(i) &= e(s_r \cdot i) \psi_r \qquad \qquad \quad \ &\text{where $s_r \cdot i$ means $i_1\dots i_{r-1} i_{r+1} i_r i_{r+2}\dots i_n$}
\end{align*}
\begin{align*} y_s y_r &= y_r y_s\\
y_s \psi_r &= \psi_r y_s &\text{if } s \neq r, r+1\\
y_r \psi_r e(i) &= \psi_r y_{r+1} e(i) &\text{if } i_r \neq i_{r+1}\\
y_r \psi_r e(i) &= (\psi_r y_{r+1} - 1) e(i) &\text{if } i_r = i_{r+1}\\
y_{r+1} \psi_r e(i) &= \psi_r y_r e(i) &\text{if } i_r \neq i_{r+1}\\
y_{r+1} \psi_r e(i) &= (\psi_r y_r + 1) e(i) &\text{if } i_r = i_{r+1}\\
&\vspace{24pt}\\
\psi_s \psi_r &= \psi_r \psi_s &\text{if } s \neq r \pm 1\\
\psi_r^2 e(i) &= 0 &\text{if } i_r = i_{r+1}\\
\psi_r^2 e(i) &= -(y_r - y_{r+1})^2 e(i) &\text{if } i_r \neq i_{r+1}\\
\psi_r \psi_{r+1} \psi_r e(i) &= \psi_{r+1} \psi_r \psi_{r+1} e(i) &\text{if $i_r = i_{r+1}$ or $i_{r+1} = i_{r+2}$}\\
\psi_r \psi_{r+1} \psi_r e(i) &= (\psi_{r+1} \psi_r \psi_{r+1} + y_r - 2y_{r+1} + y_{r+2}) e(i)&\text{if } i_r \neq i_{r+1} \neq i_{r+2}.
\end{align*}

\subsection{Specht Modules}

The Specht modules are an important family of $\mathscr{H}$-modules which are indexed by partitions of $n$. They are central in the representation theory of $\mathscr{H}$, and have their counterparts in the KLR world. Thanks to the Brundan--Kleshchev isomorphism, we will be treating $\mathscr{H}$ as a KLR algebra and as such would like to look at the Specht modules in this setting. We will use a presentation given in \cite[Definition 7.11]{kmr}. Note that in \cite{kmr}, the authors call this the column Specht module.

\begin{defn}
Let $\lambda=(\lambda_1,\lambda_2,\dots)$ be a partition of $n$. Then we define the \emph{Young diagram} of $\lambda$ to be \[[\lambda]=\{(i,j) \ | \ j\leqslant\lambda_i\}\subset \mathbb{N}^2.\] We define a \emph{node} of $\lambda$ to be an element of $[\lambda]$. Finally, we define a \emph{$\lambda$-tableau} to be the diagram $[\lambda]$ with nodes replaced by the numbers $1,2,\dots,n$ with no multiplicities.
\end{defn}

\begin{notn}
We write $\lambda\vdash n$ to mean $\lambda$ is a partition of $n$. We denote by $\lambda'$ the partition conjugate to $\lambda$. That is, the partition obtained by interchanging rows and columns of the Young diagram of $\lambda$. A tableau is said to be \emph{standard} if entries increase along the rows and down the columns. We denote the set of standard $\lambda$-tableaux by $\std(\lambda)$. If $T$ is a $\lambda$-tableau, we write $j\downarrow_T j+1$ to mean that $j$ and $j+1$ are in the same column of $T$ and that $j+1$ is in a lower row than $j$ in $T$.
\end{notn}

\begin{defn}
Let $\lambda\vdash n$ and $T\in\std(\lambda)$. Assign each node $(i,j)$ of $\lambda$ a \emph{residue} $j-i\pmod{e}$. Denote by $i_k=i_k(T)$ the residue of the node in $T$ occupied by $k$. Then we define the residue sequence of $T$ as $i_T=(i_1,i_2,\dots,i_n)$.
\end{defn}

\begin{defn}
Denote by \emph{$T_\lambda$} the standard $\lambda$-tableau with entries written consecutively down the columns. We call $T_\lambda$ the initial tableau. Denote the residue sequence of $T_\lambda$ by \emph{$i_\lambda$}.
\end{defn}
\begin{eg}
Let $\lambda=(3,2,1)$ and $e=2$. Then $T_\lambda=\young(146,25,3)$ and $i_\lambda=(010100)$.
\end{eg}

We now introduce the Specht modules. Since we will only be interested in those Specht modules indexed by hook partitions, we can give a fairly simple looking presentation. For general partitions, the Garnir relations are much more complicated.

\begin{defn}
Given $\lambda\vdash n$, the Specht module $S_\lambda$ is the module generated by $z_\lambda$ subject to the following relations:
\begin{enumerate}
\item $y_kz_\lambda=0$ for all $k$;
\item $e(i)z_\lambda=\delta_{i,i_\lambda}z_\lambda$ for all $i\in\{0,1\}^n$;
\item $\psi_jz_\lambda=0$ for all $j=1,\dots,n-1$ such that $j\downarrow_{T_\lambda}j+1$;
\item The homogeneous Garnir relations defined in \cite[Definition 7.11]{kmr}.
\end{enumerate}
Unless there is possible confusion, we shall write $z=z_\lambda$ for the sake of tidiness.
\end{defn}
\begin{rem}
Suppose we are looking at a hook partition $\lambda=(a,1^b)$. Then relation 3 can be rewritten $\psi_jz=0$ for all $j<b+1$. In this case, the Garnir relations in 4 are $\psi_jz=0$ for all $j>b+1$ and $\psi_1\psi_2\psi_3\dots\psi_{b+1}z=0$. So we have \[S_{(a,1^b)}=\langle z \ | \ e(i)z=\delta_{i,i_\lambda}, \ y_k z=0 \ \forall k, \ \psi_j z=0 \ \forall j\neq b+1, \ \psi_1\psi_2\dots\psi_{b+1}z=0\rangle.\]
\end{rem}

\begin{defn}
Let $\lambda \vdash n$ and $T\in\std(\lambda)$. Define $w_T\in\mathfrak{S}_n$ to be the permutation such that $w_TT_\lambda=T$. Let $s_i$ be the basic transposition $(i,i+1)$. Fix a preferred reduced expression $w=s_{r_1}\dots s_{r_m}$ for each element $w\in\mathfrak{S}_n$, where $m$ is minimal and $1\leqslant r_1,\dots,r_m\leqslant n-1$. Define $\psi_w=\psi_{s_{r_1}}\dots \psi_{s_{r_m}}$. Finally, define $v_T=\psi_{w_T}z$.
\end{defn}

\begin{thm}\cite[Corollary 7.20]{kmr}
Let $\lambda\vdash n$. Then $\mathscr{B}=\{v_T \ | \ T\in\std(\lambda)\}$ is a basis of $S_\lambda$.
\end{thm}

\begin{rem}
Note that the elements $v_T$ depend on choices of preferred expressions of elements of $\mathfrak{S}_n$. However, when $\lambda$ is a hook partition, these expressions are unique up to applying commutation relations, and hence $v_T$ is well-defined.
\end{rem}

We now look at some basic results for Specht modules which will be useful for our purposes.

\begin{thm}\label{conjugate}
$S_\lambda$ is decomposable if and only if $S_{\lambda'}$ is.
\end{thm}

\begin{proof}
The result follows from \cite[Theorem 3.5]{dj2}.
\end{proof}

\begin{thm}
If $e\neq 2$, or if $\lambda$ is $2$ regular, then $S_\lambda$ is indecomposable.
\end{thm}

\begin{proof}
The result follows from \cite[Corollary 8.7]{dj3} using a similar argument to that used by James in \cite{j} to prove the analogous result for the symmetric group from \cite[Theorem 13.13]{j}.
\end{proof}

In view of this last result, we seek to classify decomposability of Specht modules $S_\lambda$ when $\lambda$ is $2$-singular and $e=2$. We will focus on the special case where $\lambda=(a,1^b)$ for $b\geqslant 2$.

\section{Specht Modules for Hook Partitions}

Decomposability of Specht modules for hook partitions was solved by Murphy in the case of the symmetric group:

\begin{thm}\cite[Theorem 4.5]{gm}
Suppose $\Char\mathbb{F}=2$. Then $S_{(a,1^b)}$ is indecomposable if and only if $n$ is even or $a-1\equiv b\pmod{2^L}$ where $2^{L-1}\leqslant b<2^L$.
\end{thm}

Using this result, we will be able to assume $\Char\mathbb{F}\neq2$ where necessary. The following result will also reduce our workload later on.

\begin{thm}\label{branching}
Suppose $a$ is odd and $b$ is even. Then $S_{(a,1^b)}$ is decomposable if and only if $S_{(a+1,1^{b+1})}$ is.
\end{thm}

\begin{proof}
For any $r\geqslant0$ and any $i$, functors
\begin{align*}
e_i^{(r)}&:\mathscr{H}_{n+r}\text{-}\bmod\rightarrow\mathscr{H}_n\text{-}\bmod\\
f_i^{(r)}&:\mathscr{H}_n\text{-}\bmod\rightarrow\mathscr{H}_{n+r}\text{-}\bmod
\end{align*}
are introduced in \cite[Section 2.2]{bk2}. These functors are exact, and have the following property: if $M$ is a non-zero module and we let $\varepsilon_i(M):=\max\{r\ |\ e_i^{(r)}M\neq0\}$, then:

\cite[Lemma 2.12]{bk2} If $D$ is a simple module, then $e_i^{(\varepsilon_i(D))}D$ is simple.

Since $e_i^{(r)}$ is exact, we have $\varepsilon_i(D)\leqslant \varepsilon_i(M)$ when $D$ is a composition factor of $M$, and so by the above lemma we deduce that the composition length of $e_i^{(\varepsilon_i(M))}M$ is at most the composition length of $M$, with equality if and only if $\varepsilon_i(D)=\varepsilon_i(M)$ for all composition factors $D$ of $M$.

A corresponding result holds with $f_i,\varphi_i$ in place of $e_i,\varepsilon_i$.

Now consider Specht modules. By \cite[Lemma 2.4]{bk2} and \cite[Equations (7)\&(8)]{bk2}, $e_i^{(r)}$ and $f_i^{(r)}$ can be interpreted as restriction and induction, respectively, followed by projection onto particular blocks. In view of the block classification for Hecke algebras of type $A$ \cite[Theorem 2.11]{lm3} and the branching rules for induction and restriction of Specht modules (\cite[Theorem 7.4]{dj1} and \cite[Proposition 1.9]{am} respectively), we deduce that $\varepsilon_i(S_\lambda)$ is the number of removable nodes of $\lambda$ of residue $i$, and $e_i^{(\varepsilon_i(S_\lambda))}S_\lambda$ is the Specht module labelled by the partition obtained by removing these nodes.  A corresponding statement holds for $f_i$ and addable nodes.

In particular, when $e=2$, $a$ is odd and $b$ is even, let $\lambda=(a,1^b)$ and $\mu=(a+1,1^{b+1})$. Then $\varepsilon_1(S_\mu)=\varphi_1(S_\lambda)=2$, and
$e_1^{(2)}S_\mu=S_\lambda$, $f_1^{(2)}S_\lambda=S_\mu$.

In view of the above results, this means that $S_\lambda$ and $S_\mu$ have the same composition length and that $e_1^{(2)}D\neq0$ for every composition factor $D$ of $S_\mu$. Hence (again by exactness) $e_1^{(2)}N\neq0$ for every submodule $N$ of $S_\mu$. Hence if $S_\mu$ is decomposable, then so is $S_\lambda$.  The same argument the other way round shows that if $S_\lambda$ is decomposable, then so is $S_\mu$.
\end{proof}

For the remainder of the paper, we fix $\lambda=(a,1^b)$ and $n=a+b$. Recall that \[S_\lambda=\langle z \ | \ y_kz=0 \ \forall k, \ \psi_jz=0 \ \forall j\neq b+1, \ \psi_1\psi_2\dots\psi_{b+1}z=0, \ e(i)z=\delta_{i,i_\lambda}z\rangle.\]

Suppose $f$ is an $\mathscr{H}$-endomorphism of $S_\lambda$. $z=e(i_\lambda)z$, so we have $f(z)\in e(i_\lambda)S_\lambda$. Now consider the basis $\mathscr{B}$. \cite[Lemma 4.4]{bkw} tells us that $e(i)v_T=\delta_{i,i_T}v_T$ for any $T\in\std(\lambda)$. Hence $\mathscr{B}\cap e(i)S_\lambda=\{v_T \ | \ i_T=i\}$. In particular, $f(z)$ is a linear combination of elements in $\mathscr{D}:=\mathscr{B}\cap e(i_\lambda)S_\lambda=\{v_T \ | \ i_T=i_\lambda\}$.
This is at the core of our approach to understanding $\End_\mathscr{H}(S_\lambda)$. 

\begin{defn}
When $\lambda=(a,1^b)$, we define the \emph{arm} to be the set of nodes

\noindent$\{(1,2),(1,3),\dots,(1,a)\}$ of $\lambda$ and the \emph{leg} to be the set of nodes $\{(2,1),(3,1),\dots,(b+1,1)\}$.
\end{defn}

\begin{lem}\label{D_i}
Suppose $b$ is even and $v_T\in\mathscr{D}$. Then for all $1\leqslant i\leqslant \lceil n/2\rceil-1$, $2i+1$ appears directly after $2i$ in $T$. That is, if $2i$ is in the leg of $T$ then $2i+1$ is directly below it, and if $2i$ is in the arm of $T$ then $2i+1$ is directly to the right of it.
\end{lem}

\begin{proof}
In defining $i_\lambda$, we assign all nodes of $\lambda$ in which $T_\lambda$ contains an even entry a $1$ and all others a $0$. First, we note that since $b$ is even and $v_T\in\mathscr{D}$, the final node in the leg of $\lambda$ has residue $0$. This ensures that if $2i$ is in the leg of $T$ there must be some entry immediately below it.

By induction on $i>1$, assume that $2i+1$ appears directly after $2i$ in $T$, for all $i<k$. Suppose our assertion is false for $i=k$. We assume without loss of generality that $2k$ is in the leg of $T$ and $2k+1$ is in the arm. Now by induction any even number, $2j<2k$, is immediately followed by $2j+1$. This forces $2k+1$ to be adjacent to $2j+1$ for some $j<k$, and $v_T\notin\mathscr{D}$.
\end{proof}

The fact that entries must stick together in these pairs motivates our next definition.

\begin{defn}
We will call the pair of entries $2i,2i+1$ for $1\leqslant i\leqslant\lceil n/2\rceil-1$ a \emph{domino}. We will denote the domino by $[2i,2i+1]$ or $D_i$. We define a \emph{domino tableau} to be any $\lambda$-tableau $T$ such that $v_T\in\mathscr{D}$. We denote the set of domino tableaux by $\Dom(\lambda)$.
\end{defn}

\begin{rem}
$\mathscr{D}=\{v_T \ | \ T\in\Dom(\lambda)\}$ is a basis of $e(i_\lambda)S_\lambda$.
\end{rem}

Now, we separate our problem into cases where $a$ and $b$ are odd or even. When $b$ is even, we have $i_\lambda = 0101\dots 01$. If $b$ is odd, however, we have $i_\lambda=0101\dots 011010\dots 10$, where we have a repetition in the positions $b+1$ and $b+2$.

We will now begin by solving the simplest cases, where $n$ is even.

\section{Decomposability of $S_{(a,1^b)}$ when $n$ is even}

First, we will look at the case where $a$ and $b$ are both even.

\begin{lem}\label{killvt}
Suppose $T\in\std(\lambda)$ and $1<i<n$. If $i,i+1,\dots,n$ all lie in the arm of $T$ then $\psi_iv_T=0$. If $i$ lies in the leg of $T$ and $i+1$ lies in the arm, then $\psi_iv_T=v_U$, where $U$ is obtained from $T$ by swapping $i$ and $i+1$.
\end{lem}

\begin{proof}
First, suppose $i,i+1,\dots,n$ all lie in the arm of $T$ for some $1<i<n$. Then $v_T$ cannot possibly involve $\psi_j$ for any $j>i-2$. It follows that $\psi_i$ commutes with each generator $\psi_j$ appearing in $v_T$ and the result follows from the Specht module relations. 

To prove the second part of the lemma, we note that $w_T^{-1}(i)<w_T^{-1}(i+1)$. This is easily seen since $w_T^{-1}(j)$ is the number that occupies the same node in $T_\lambda$ that $j$ occupies in $T$. Hence if $s_{i_1}s_{i_2}\dots s_{i_r}$ is a reduced expression for $w_T$, then $s_is_{i_1}s_{i_2}\dots s_{i_r}$ is a reduced expression for $s_iw_T$. So $\psi_iv_T=v_U$.
\end{proof}

\begin{thm}\label{a/b evens}
If $a$ and $b$ are both even, then $\End_\mathscr{H}(S_\lambda)$ is one-dimensional. In particular, $S_\lambda$ is indecomposable.
\end{thm}

\begin{proof}
Suppose $f\in\End_\mathscr{H}(S_\lambda)$. Then by the above remark, \[f(z)=\sum_{T\in\Dom(\lambda)}\alpha_Tv_T\text{ for some } \alpha_T\in\mathbb{F}.\] Then by Lemma \ref{killvt}, acting on the left by $\psi_{n-1}$ annihilates all $v_T$ for tableaux $T$ which don't have $D_{\frac{n-2}{2}}$ in their leg. Now, in the remaining tableaux, Lemma \ref{killvt} gives us that $\psi_{n-1}v_T\in\mathscr{B}\setminus\mathscr{D}$. Since $\psi_{n-1}f(z)=0$, we must have $\alpha_T=0$ for all $T$ which have $D_{\frac{n-2}{2}}$ in the leg.

In this way, we act on $f(z)$ by $\psi_{n+1-2i}$ for $i=1,2,\dots,(a-2)/2$ to annihilate all $v_T$ for tableaux $T$ which don't have $D_{\frac{n-2i}{2}}$ in the leg. At each step, we deduce that $\alpha_T=0$ if $T$ has $D_{\frac{n-2i}{2}}$ in the leg.

Therefore $f(z)=\alpha z$ for some $\alpha\in\mathbb{F}$ and the result follows.
\end{proof}

Next, we look at the case where $a$ and $b$ are both odd.

\begin{thm}\label{a/b odds}
If $a$ and $b$ are both odd, then $\End_\mathscr{H}(S_\lambda)$ is one-dimensional. In particular, $S_\lambda$ is indecomposable.
\end{thm}

\begin{proof}
The result follows from Theorem \ref{a/b evens} by application of Theorem \ref{conjugate}.
\end{proof}

\section{KLR actions on $\mathscr{D}$ when $n$ is odd}

When $n$ is odd, much more work must be done. By Theorem \ref{conjugate}, we can assume throughout this section that $b<n/2$.

Using Theorem \ref{branching}, we can focus on the case where $a$ is odd and $b$ is even, as it is slightly easier to work with. The case where $a$ is even and $b$ is odd will then follow.

Recall that $\mathscr{D}=\{v_T \ | \ T\in\Dom(\lambda)\}$ is a basis of $e(i_\lambda)S_\lambda$. At this point we introduce some new notation which is much needed to keep things tidy!

\begin{defn}
We define $\Psi_j := \psi_j \psi_{j+1} \psi_{j-1} \psi_j$. For $3\leqslant x\leqslant y\leqslant n-2$ two odd integers, we then define:

\[\Psichaind{y}{x}:=\Psi_y\Psi_{y-2}\dots\Psi_x\quad\text{and}\quad\Psichainu{x}{y}:=\Psi_x\Psi_{x+2}\dots\Psi_{y}.\]
If $y<x$ we consider both of the above defined terms to be the identity element of our field.
\end{defn}

\begin{rem}
Given some $T\in\Dom(\lambda)$, let $2d$ be the number of entries in the leg of $T$ which differ from the entries in the corresponding nodes of $T_\lambda$. Notice that these will consist of the final $d$ dominoes in the leg, since $T\in\std(\lambda)$.

Let $j_1',\dots,j_d'$ be the odd numbers (in ascending order) in the $d$ dominoes in the leg of $T$ which differ from the corresponding entries in $T_\lambda$ and define $j_i:=j_i'-2$ for each $i$. For example, if $\lambda=(7,1^6)$ then \[T_\lambda=\young(189\ten\eleven\twelve\thirteen,2,3,4,5,6,7).\quad\text{ Let }\quad T=\young(14567\ten\eleven,2,3,8,9,\twelve,\thirteen).\]
Then $d=2$ and we see that $j_1=7$ and $j_2=11$. 

Now, we can see that $v_T$ can be written as the reduced expression \[\Psichaind{j_1}{b+3-2d}\Psichaind{j_2}{b+5-2d}\dots\Psichaind{j_d}{b+1}z.\] We will refer to this as the \emph{normal form} for $v_T$. Notice that $j_{i+1}>j_i$ for all $i=1,\dots,d-1$. It will be useful to note that if $v_T\in\mathscr{D}$ is in our normal form, then any expression obtained from it by deleting $\Psi$ terms from the left is also an element in $\mathscr{D}$.
\end{rem}

\begin{defn}
Let $T\in\Dom(\lambda)$. We define the length $r(T)$ of $T$ to be the number of $\Psi$ terms in the normal form of $v_T$.
\end{defn}

In the next three results, we examine the actions of the generators of $\mathscr{H}$ on the elements of $\mathscr{D}$.

\begin{lem}\label{commuting}
\begin{align*}
e(i_\lambda)\Psi_j&=\Psi_je(i_\lambda) & \text{for all $j$,}\\
y_k\Psi_j&=\Psi_jy_k & \text{for all $k\geqslant j+3$ and for all $k\leqslant j-2$,}\\
\psi_k\Psi_j&=\Psi_j\psi_k & \text{for all $k\geqslant j+3$ and for all $k\leqslant j-3$.}
\end{align*}
\end{lem}

\begin{proof}
Clear from the definition of $\Psi_j$ and the defining relations.
\end{proof}

\begin{prop} \label{y-actions}
Suppose $T\in\Dom(\lambda)$. Then
\begin{align*}
y_kv_T&=0\text{ for all $k$}; & (A)\\
\psi_kv_T&=0\text{ for all even $k$}; & (B)\\
\psi_1v_T&=0. & (C)
\end{align*}
\end{prop}

\begin{proof}
Let $(A_r)$ denote the statement that $(A)$ holds for all $T$ with $r(T)=r$, and define $(B_r)$ similarly. We first prove $(A_r)$ and $(B_r)$ simultaneously, by induction on $r$.

First we must show that $(A_0)$ and $(B_0)$ hold. In this case, $v_T=z$ and the defining relations give our result immediately.

Now, let $v_T=\Psichaind{j_1}{b+3-2d}\Psichaind{j_2}{b+5-2d}\dots\Psichaind{j_d}{b+1}z$ be in normal form for some $d$, and define \[v_{T_{(2)}}:=\Psichaind{j_2}{b+5-2d}\dots\Psichaind{j_d}{b+1}z.\]

We will show that $(A_{r-1}) \& (B_{r-1})\Rightarrow (A_r)$. We split our problem into 5 cases:

\begin{enumerate}
\item $k=j_1+2$,
\item $k=j_1+1$,
\item $k=j_1$,
\item $k=j_1-1$,
\item All other $k$.
\end{enumerate}

We can now solve each case quite simply!

\begin{enumerate}
{\item\abovedisplayskip=0pt\abovedisplayshortskip=0pt~\vspace*{-\baselineskip}
\begin{align*}
y_{j_1+2}v_T&=\psi_{j_1}\psi_{j_1-1}(y_{j_1+2}\psi_{j_1+1}e(s_{j_1}\cdot i_\lambda))\psi_{j_1}\Psichaind{j_1-2}{b+3-2d}v_{T_{(2)}}\\
&=\psi_{j_1}\psi_{j_1-1}(\psi_{j_1+1}y_{j_1+1}+1)\psi_{j_1}\Psichaind{j_1-2}{b+3-2d}v_{T_{(2)}}\\
&=\Psi_{j_1}\underbrace{y_{j_1}\Psichaind{j_1-2}{b+3-2d}v_{T_{(2)}}}_{=0\text{ by }(A_{r-1})} +\psi_{j_1}\psi_{j_1-1}\psi_{j_1}\Psichaind{j_1-2}{b+3-2d}v_{T_{(2)}}\\
&=(\psi_{j_1-1}\psi_{j_1}\psi_{j_1-1}-y_{j_1-1}+2y_{j_1}-y_{j_1 +1})\Psichaind{j_1-2}{b+3-2d}v_{T_{(2)}}\\
&=\psi_{j_1-1}\psi_{j_1}\underbrace{\psi_{j_1-1}\Psichaind{j_1-2}{b+3-2d}v_{T_{(2)}}}_{=0\text{ by } (B_{r-1})}-\underbrace{y_{j_1-1}\Psichaind{j_1-2}{b+3-2d}v_{T_{(2)}}}_{=0\text{ by }(A_{r-1})}\\
&\quad+2\underbrace{y_{j_1}\Psichaind{j_1-2}{b+3-2d}v_{T_{(2)}}}_{=0\text{ by }(A_{r-1})}-\underbrace{y_{j_1+1}\Psichaind{j_1-2}{b+3-2d}v_{T_{(2)}}}_{=0\text{ by }(A_{r-1})}\\
&=0.
\end{align*}}
{\item\abovedisplayskip=0pt\abovedisplayshortskip=0pt~\vspace*{-\baselineskip}
\begin{align*}
y_{j_1+1}v_T&=(y_{j_1+1}\psi_{j_1}e(s_{j_1}\cdot i_\lambda))\psi_{j_1+1}\psi_{j_1-1}\psi_{j_1}\Psichaind{j_1-2}{b+3-2d}v_{T_{(2)}}\\
&=\psi_{j_1}\psi_{j_1+1}(y_{j_1}\psi_{j_1-1}e(s_{j_1}\cdot i_\lambda))\psi_{j_1}\Psichaind{j_1-2}{b+3-2d}v_{T_{(2)}}\\
&= \psi_{j_1}\psi_{j_1+1}(\psi_{j_1-1}y_{j_1-1}+1)\psi_{j_1}\Psichaind{j_1-2}{b+3-2d}v_{T_{(2)}}\\
&= \Psi_{j_1}\underbrace{y_{j_1-1}\Psichaind{j_1-2}{b+3-2d}v_{T_{(2)}}}_{=0 \text{ by }(A_{r-1})}+\psi_{j_1}\psi_{j_1+1}\psi_{j_1}\Psichaind{j_1-2}{b+3-2d}v_{T_{(2)}}\\
&=(\psi_{j_1+1}\psi_{j_1}\psi_{j_1+1}+y_{j_1}-2y_{j_1+1}+y_{j_1+2})\Psichaind{j_1-2}{b+3-2d}v_{T_{(2)}}\\
&=0\quad\text{by $(A_{r-1})$ and $(B_{r-1})$.}
\end{align*}}
{\item\abovedisplayskip=0pt\abovedisplayshortskip=0pt~\vspace*{-\baselineskip}
\begin{align*}
y_{j_1}v_T&=(y_{j_1}\psi_{j_1}e(s_{j_1}\cdot i_\lambda))\psi_{j_1+1}\psi_{j_1-1}\psi_{j_1}\Psichaind{j_1-2}{b+3-2d}v_{T_{(2)}}\\
&=\psi_{j_1}(y_{j_1+1}\psi_{j_1+1}e(s_{j_1}\cdot i_\lambda))\psi_{j_1-1}\psi_{j_1}\Psichaind{j_1-2}{b+3-2d}v_{T_{(2)}}\\
&=\psi_{j_1}(\psi_{j_1+1}y_{j_1+2}-1)\psi_{j_1-1}\psi_{j_1}\Psichaind{j_1-2}{b+3-2d}v_{T_{(2)}}\\
&=\Psi_{j_1}\underbrace{y_{j_1+2}\Psichaind{j_1-2}{b+3-2d}v_{T_{(2)}}}_{=0\text{ by }(A_{r-1})}\\
&\quad-(\psi_{j_1-1}\psi_{j_1}\psi_{j_1-1}-y_{j_1-1}+2y_{j_1}-y_{j_1+1})\Psichaind{j_1-2}{b+3-2d}v_{T_{(2)}}\\
&=0\quad\text{by $(A_{r-1})$ and $(B_{r-1})$.}
\end{align*}}
{\item\abovedisplayskip=0pt\abovedisplayshortskip=0pt~\vspace*{-\baselineskip}
\begin{align*}
y_{j_1-1}v_T&=\psi_{j_1}\psi_{j_1+1}(y_{j_1-1}\psi_{j_1-1}e(s_{j_1}\cdot i_\lambda))\psi_{j_1}\Psichaind{j_1-2}{b+3-2d}v_{T_{(2)}}\\
&=\psi_{j_1}\psi_{j_1+1}(\psi_{j_1-1}y_{j_1}-1)\psi_{j_1}\Psichaind{j_1-2}{b+3-2d}v_{T_{(2)}}\\
&=\Psi_{j_1}\underbrace{y_{j_1+1}\Psichaind{j_1-2}{b+3-2d}v_{T_{(2)}}}_{=0\text{ by }(A_{r-1})}\\
&\quad-(\psi_{j_1+1}\psi_{j_1}\psi_{j_1+1}+y_{j_1}-2y_{j_1+1}+y_{j_1+2})\Psichaind{j_1-2}{b+3-2d}v_{T_{(2)}}\\
&=0\quad\text{by $(A_{r-1})$ and $(B_{r-1})$.}
\end{align*}}
\item Now suppose $k\neq j_1+2,j_1+1,j_1$ or $j_1-1$. Then
\begin{align*}
y_k v_T&=\Psi_{j_1}y_k\Psichaind{j_1-2}{b+3-2d}v_{T_{(2)}}\quad\text{by Lemma \ref{commuting}}\\
&=0\quad\text{by $(A_{r-1})$.}
\end{align*}
\end{enumerate}

Next, we show that $(A_{r-1})\&(B_{r-1})\Rightarrow(B_r)$. Once again we split this into the following cases:

\begin{enumerate}
\item $k=j_1+1$,
\item $k=j_1-1$,
\item All other $k$.
\end{enumerate}

\begin{enumerate}
{\item\abovedisplayskip=0pt\abovedisplayshortskip=0pt~\vspace*{-\baselineskip}
\begin{align*}
\psi_{j_1+1}v_T&=(\psi_{j_1+1}\psi_{j_1}\psi_{j_1+1}e(s_{j_1}\cdot i_\lambda))\psi_{j_1-1}\psi_{j_1}\Psichaind{j_1-2}{b+3-2d}v_{T_{(2)}}\\
&=\psi_{j_1}\psi_{j_1+1}(\psi_{j_1}\psi_{j_1-1}\psi_{j_1}e(i_\lambda))\Psichaind{j_1-2}{b+3-2d}v_{T_{(2)}}\\
&=\psi_{j_1}\psi_{j_1+1}(\psi_{j_1-1}\psi_{j_1}\psi_{j_1-1}-y_{j_1-1}+2 y_{j_1}-y_{j_1+1})\Psichaind{j_1-2}{b+3-2d}v_{T_{(2)}}\\
&=0\quad\text{by $(A_{r-1})$ and $(B_{r-1})$.}
\end{align*}}
{\item\abovedisplayskip=0pt\abovedisplayshortskip=0pt~\vspace*{-\baselineskip}
\begin{align*}
\psi_{j_1-1}v_T&=(\psi_{j_1-1}\psi_{j_1}\psi_{j_1-1}e(s_{j_1}\cdot i_\lambda))\psi_{j_1+1}\psi_{j_1}\Psichaind{j_1-2}{b+3-2d}v_{T_{(2)}}\\
&=\psi_{j_1}\psi_{j_1-1}(\psi_{j_1}\psi_{j_1+1}\psi_{j_1}e(i_\lambda))\Psichaind{j_1-2}{b+3-2d}v_{T_{(2)}}\\
&=\psi_{j_1}\psi_{j_1-1}(\psi_{j_1+1}\psi_{j_1}\psi_{j_1+1}+y_{j_1}-2y_{j_1+1}+y_{j_1+2})\Psichaind{j_1-2}{b+3-2d}v_{T_{(2)}}\\
&=0\quad\text{by $(A_{r-1})$ and $(B_{r-1})$.}
\end{align*}}
\item Now suppose $k$ is even but $k\neq j_1+1$ or $j_1-1$. Then
\begin{align*}
\psi_kv_T&=\Psi_{j_1}\psi_k\Psichaind{j_1-2}{b+3-2d}v_{T_{(2)}}\quad\text{by Lemma \ref{commuting}}\\
&=0\quad\text{by $(B_{r-1})$.}
\end{align*}
\end{enumerate}
And so, our results follow.

Now, we prove $(C)$. If $d<b/2$, then $\Psi_3$ does not occur in $v_T$, and so $\psi_1$ commutes with all $\Psi$ terms in $v_T$ and the result is clear. So suppose $d=b/2$. Then $v_T=\Psichaind{j_1}{3}\Psichaind{j_2}{5}\dots\Psichaind{j_d}{b+1}z$. It's easy to see that \[\psi_1\psi_2\dots\psi_i\Psichaind{j_{(i+2d-b+1)/2}}{i+2}=\Psichaind{j_{(i+2d-b+1)/2}}{i+4}\psi_{i+2}\psi_{i+3}(\psi_1\psi_2\dots\psi_{i+2}).\] 

Applying this for $i=1,3,\dots,b-1$ in turn, we obtain \[\psi_1v_T=\Psichaind{j_1}{5}\Psichaind{j_2}{7}\dots\Psichaind{j_d}{b+3}\psi_3\psi_4\dots\psi_{b+2}\psi_1\psi_2\dots\psi_{b+1}z,\] which is zero in view of the Garnir relation $\psi_1\psi_2\dots\psi_{b+1}z=0$.\qedhere
\end{proof}

\begin{rem}
A shorter but less direct proof of $(A_r)$ and $(B_r)$ can be given using the grading on $\mathscr{H}$ and $S_\lambda$, closely mimicking the proof of \cite[Lemma 4.4]{kmr}.
\end{rem}

\begin{lem}\label{psi-actions}
Suppose $j$ is odd and $T\in\Dom(\lambda)$. Then

\begin{enumerate}
\item$\psi_j\Psi_jv_T=-2\psi_jv_T$,
\item$\psi_{j}\Psi_{j+2}\Psi_{j}v_T=\psi_jv_T$,
\item$\Psi_j\psi_{j+2}v_T=0$,
\item$\psi_j\Psi_{j-2}\Psi_jv_T=\psi_jv_T$,
\item$\Psi_j\psi_{j-2}v_T=0$.
\end{enumerate}
\end{lem}

\needspace{3em}
\begin{proof}\indent
\begin{enumerate}
\vspace{-\topsep}
\Item
\begin{align*}
\psi_j\Psi_je(i_\lambda)v_T&=(\psi_j^2 e(s_j \cdot i_\lambda))\psi_{j+1}\psi_{j-1}\psi_jv_T\\
&=(-y_j^2-y_{j+1}^2+2y_jy_{j+1})\psi_{j+1}\psi_{j-1}\psi_jv_T\\
&=-\psi_{j+1}y_j(y_j\psi_{j-1}e(s_j\cdot i_\lambda))\psi_jv_T\\
&\quad-y_{j+1}(y_{j+1}\psi_{j+1}e(s_j\cdot i_\lambda))\psi_{j-1}\psi_jv_T\\
&\quad+2y_j(y_{j+1}\psi_{j+1}e(s_j\cdot i_\lambda))\psi_{j-1}\psi_jv_T\\
&=-\psi_{j+1}y_j(\psi_{j-1}y_{j-1}+1)\psi_jv_T\\
&\quad-y_{j+1}(\psi_{j+1}y_{j+2}-1)\psi_{j-1}\psi_jv_T\\
&\quad+2y_j(\psi_{j+1}y_{j+2}-1)\psi_{j-1}\psi_jv_T\\
&=-\psi_{j+1}y_j\psi_jv_T+\psi_{j-1}y_{j+1}\psi_jv_T-2(y_j\psi_{j-1}e(s_j\cdot i_\lambda))\psi_j v_T\\
&=-\psi_{j+1}\psi_jy_{j+1}v_T+\psi_{j-1}\psi_jy_jv_T-2(\psi_{j-1}y_{j-1}+1)\psi_jv_T\\
&=-2\psi_jv_T.
\end{align*}
\item[2.,3.,4.\& 5.] We have \begin{align*}
\psi_j\Psi_{j+2}\Psi_je(i_\lambda)v_T&=\psi_{j+2}\psi_{j+3}(\psi_j\psi_{j+1}\psi_je(s_{j+2}\cdot s_j\cdot i_\lambda))\psi_{j+2}\psi_{j+1}\psi_{j-1}\psi_j\\
&=\psi_{j+2}\psi_{j+3}(\psi_{j+1}\psi_j\psi_{j+1}+y_j-2y_{j+1}+y_{j+2})\psi_{j+2}\psi_{j+1}\psi_{j-1}\psi_jv_T\\
&=\psi_{j+2}\psi_{j+3}\psi_{j+1}\psi_j(\psi_{j+1}\psi_{j+2}\psi_{j+1}e(s_j \cdot i_\lambda))\psi_{j-1}\psi_jv_T\\
&\quad+\psi_{j+2}\psi_{j+3}\psi_{j+2}\psi_{j+1}(y_j\psi_{j-1}e(s_j\cdot i_\lambda))\psi_jv_T\\
&\quad-2\psi_{j+2}\psi_{j+3}\psi_{j+2}(y_{j+1}\psi_{j+1}e(s_j\cdot i_\lambda))\psi_{j-1}\psi_jv_T\\
&\quad+\psi_{j+2}\psi_{j+3}(y_{j+2}\psi_{j+2}e(s_j\cdot i_\lambda))\psi_{j+1}\psi_{j-1}\psi_j v_T\\
&=\psi_{j+2}\psi_{j+3}\psi_{j+1}\psi_j\psi_{j+2}\psi_{j+1}\psi_{j+2}\psi_{j-1}\psi_jv_T\\
&\quad+\psi_{j+2}\psi_{j+3}\psi_{j+2}\psi_{j+1}(\psi_{j-1}y_{j-1}+1)\psi_jv_T\\
&\quad-2\psi_{j+2}\psi_{j+3}\psi_{j+2}(\psi_{j+1}y_{j+2}-1)\psi_{j-1}\psi_jv_T\\
&\quad+\psi_{j+2}\psi_{j+3}\psi_{j+2}y_{j+3}\psi_{j+1}\psi_{j-1}\psi_j v_T\\
&=\Psi_{j+2}\Psi_j\psi_{j+2}v_T+\psi_{j+2}\psi_{j+3}\psi_{j+2}\psi_{j+1}\psi_{j-1}\psi_j\underbrace{y_{j-1}v_T}_{=0}\\
&\quad+(\psi_{j+2}\psi_{j+3}\psi_{j+2}e(s_j\cdot i_\lambda))\psi_{j+1}\psi_jv_T\\
&\quad-2\psi_{j+2}\psi_{j+3}\psi_{j+2}\psi_{j+1}\psi_{j-1}\psi_j\underbrace{y_{j+2}v_T}_{=0}\\
&\quad+2(\psi_{j+2}\psi_{j+3}\psi_{j+2}e(s_j\cdot i_\lambda))\psi_{j-1}\psi_jv_T\\
&\quad+\psi_{j+2}\psi_{j+3}\psi_{j+2}\psi_{j+1}\psi_{j-1}\psi_j\underbrace{y_{j+3}v_T}_{=0}\\
&=\Psi_{j+2}\Psi_j\psi_{j+2}v_T+(\psi_{j+3}\psi_{j+2}\psi_{j+3}+y_{j+2}-2y_{j+3}+y_{j+4})\psi_{j+1}\psi_jv_T\\
&\quad+2(\psi_{j+3}\psi_{j+2}\psi_{j+3}+y_{j+2}-2y_{j+3}+y_{j+4})\psi_{j-1}\psi_jv_T\\
&=\Psi_{j+2}\Psi_j\psi_{j+2}v_T+\psi_{j+3}\psi_{j+2}\psi_{j+1}\psi_j\underbrace{\psi_{j+3}v_T}_{=0}\\
&\quad+(y_{j+2}\psi_{j+1}e(s_j\cdot i_\lambda))\psi_jv_T+2\psi_{j+3}\psi_{j+2}\psi_{j-1}\psi_j\underbrace{\psi_{j+3}v_T}_{=0}\\
&=\Psi_{j+2}\Psi_j\psi_{j+2}v_T+(\psi_{j+1}y_{j+1}+1)\psi_jv_T\\
&=\Psi_{j+2}\Psi_j\psi_{j+2}v_T+\psi_jv_T+\psi_{j+1}y_{j+1}\psi_jv_T\\
&=\Psi_{j+2}\Psi_j\psi_{j+2}v_T+\psi_jv_T+\psi_{j+1}\psi_j\underbrace{y_jv_T}_{=0}.
\end{align*}

We also have \begin{align*}
\psi_j\Psi_{j-2}\Psi_je(i_\lambda)v_T&=\psi_{j-2}\psi_{j-3}(\psi_j\psi_{j-1}\psi_j e(s_{j-2}\cdot s_j\cdot i_\lambda))\psi_{j-2}\psi_{j+1}\psi_{j-1}\psi_jv_T\\
&=\psi_{j-2}\psi_{j-3}(\psi_{j-1}\psi_j\psi_{j-1}-y_{j-1}+2y_j-y_{j+1})\psi_{j-2}\psi_{j+1}\psi_{j-1}\psi_jv_T\\
&=\psi_{j-2}\psi_{j-3}\psi_{j-1}\psi_j(\psi_{j-1}\psi_{j-2}\psi_{j-1}e(s_j \cdot i_\lambda))\psi_{j+1}\psi_jv_T\\
&\quad-\psi_{j-2}\psi_{j-3}(y_{j-1}\psi_{j-2}e(s_j \cdot i_\lambda))\psi_{j+1}\psi_{j-1}\psi_jv_T\\
&\quad+2\psi_{j-2}\psi_{j-3}\psi_{j-2}\psi_{j+1}(y_j \psi_{j-1}e(s_j \cdot i_\lambda))\psi_jv_T\\
&\quad-\psi_{j-2}\psi_{j-3}\psi_{j-2}(y_{j+1}\psi_{j+1}e(s_j \cdot i_\lambda))\psi_{j-1}\psi_jv_T\\
&=\psi_{j-2}\psi_{j-3}\psi_{j-1}\psi_j\psi_{j-2}\psi_{j-1}\psi_{j-2}\psi_{j+1}\psi_jv_T\\
&\quad-\psi_{j-2}\psi_{j-3}\psi_{j-2}y_{j-2}\psi_{j+1}\psi_{j-1}\psi_jv_T\\
&\quad+2\psi_{j-2}\psi_{j-3}\psi_{j-2}\psi_{j+1}(\psi_{j-1}y_{j-1}+1)\psi_jv_T\\
&\quad-\psi_{j-2}\psi_{j-3}\psi_{j-2}(\psi_{j+1}y_{j+2}-1)\psi_{j-1}\psi_jv_T\\
&=\Psi_{j-2}\Psi_j\psi_{j-2}v_T-\psi_{j-2}\psi_{j-3}\psi_{j-2}\psi_{j+1}\psi_{j-1}\psi_j\underbrace{y_{j-2}v_T}_{=0}\\
&\quad+2(\psi_{j-2}\psi_{j-3}\psi_{j-2}e(s_j\cdot i_\lambda))\psi_{j+1}\psi_jv_T\\
&\quad+(\psi_{j-2}\psi_{j-3}\psi_{j-2}e(s_j\cdot i_\lambda))\psi_{j-1}\psi_jv_T\\
&=\Psi_{j-2}\Psi_j\psi_{j-2}v_T\\
&\quad+2(\psi_{j-3}\psi_{j-2}\psi_{j-3}-y_{j-3}+2y_{j-2}-y_{j-1})\psi_{j+1}\psi_jv_T\\
&\quad+(\psi_{j-3}\psi_{j-2}\psi_{j-3}-y_{j-3}+2 y_{j-2}-y_{j-1})\psi_{j-1}\psi_jv_T\\
&=\Psi_{j-2}\Psi_j\psi_{j-2}v_T+2(0)-(y_{j-1}\psi_{j-1}e(s_j\cdot i_\lambda))\psi_jv_T\\
&=\Psi_{j-2}\Psi_j\psi_{j-2}v_T-(\psi_{j-1}y_j-1)\psi_jv_T\\
&=\Psi_{j-2}\Psi_j\psi_{j-2}v_T-0+\psi_jv_T.
\end{align*}

So we have
\begin{align*}
\psi_{j}\Psi_{j+2}\Psi_{j}v_T&=\psi_jv_T+\Psi_{j+2}\Psi_j\psi_{j+2}v_T & (*)\\
\psi_j\Psi_{j-2}\Psi_jv_T&=\psi_jv_T+\Psi_{j-2}\Psi_j\psi_{j-2}v_T. & (**)
\end{align*}

Now all four statements will follow if we can show that 3 and 5 hold. We will proceed by proving both simultaneously by induction on $r(T)$. That is, we will prove that
\begin{align*}
\Psi_j\psi_{j+2}v_T&=0\text{ for any odd $j$ and }r(T)=r, & (A_r)\\
\Psi_j\psi_{j-2}v_T&=0\text{ for any odd $j$ and }r(T)=r, & (B_r)
\end{align*}
by simultaneous induction on $r$.

First, we prove that $(A_r)$ follows if $(A_s)$ and $(B_s)$ hold for all $s<r$.

$(A_0)$ is clearly true. We have $\Psi_j\psi_{j+2}z=\psi_j\psi_{j+1}\psi_{j-1}\psi_{j+2}\psi_jz=0$ since at least one of $\psi_j$, $\psi_{j+2}$ must annihilate $z$.

Now let $r>0$. Suppose $v_T=\Psichaind{j_1}{b+3-2d}\Psichaind{j_2}{b+5-2d}\dots\Psichaind{j_d}{b+1}z$ is in normal form and define $v_{T'}:=\Psichaind{j_1-2}{b+3-2d}\Psichaind{j_2}{b+5-2d}\dots\Psichaind{j_d}{b+1}z$.

If $j_1\geqslant j+6$ or $j_1\leqslant j-4$, then we clearly have $\Psi_j\psi_{j+2}v_T=\Psi_{j_1}\Psi_j\psi_{j+2}v_{T'}$ and our result follows by $(A_{r-1})$. So we break our proof up for the remaining four possibilities.
\begin{enumerate}
\item Suppose $j_1=j+4$. We will write $v_{T_{(2)}}:=\Psichaind{j_2}{b+5-2d}\dots\Psichaind{j_d}{b+1}z$. If $b+3-2d=j+4$ also, we have
\begin{align*}
\Psi_j\psi_{j+2}v_T&=\Psi_j\psi_{j+2}\Psi_{j+4}v_{T_{(2)}}\\
&=0,
\end{align*}

as we have a $\psi_j$ which commutes with everything to its right, given that the lowest indexed $\Psi$-term in $v_{T_{(2)}}$ is $\Psi_{j+6}$.

If $b+3-2d<j+4$ we have
\begin{align*}
\Psi_j (\psi_{j+2}\Psi_{j+4}\Psi_{j+2})\Psichaind{j}{b+3-2d}v_{T_{(2)}}&=\Psi_j(\psi_{j+2}+\Psi_{j+4}\Psi_{j+2}\psi_{j+4})\Psichaind{j}{b+3-2d}v_{T_{(2)}}\text{ by $(*)$}\\
&=0\quad\text{by $(A_s)$ for some $s<r$, as $\Psichaind{j}{b+3-2d}v_{T_{(2)}}\in\mathscr{D}$}.
\end{align*}
\item Suppose $j_1=j+2$. Then we have
\begin{align*}
\Psi_j\psi_{j+2}\Psi_{j+2}v_{T'}&=-2\Psi_j\psi_{j+2}v_{T'}\text{ by part 1,}\\
&=0\quad\text{by $(A_{r-1})$, as $r(T') = r-1$.}
\end{align*}
\item Suppose $j_1=j$. Then we have
\begin{align*}
\Psi_j\psi_{j+2}\Psi_jv_{T'}&=-2\Psi_j\psi_{j+2}v_{T'}\text{ by part 1,}\\
&=0\text{ by $(A_{r-1})$.}
\end{align*}
\item Suppose $j_1=j-2$. We will write $v_{T_{(3)}}:=\Psichaind{j_3}{b+7-2d}\dots\Psichaind{j_d}{b+1}z$. Here, we must divide into further subcases.
\begin{enumerate}
\item Suppose $j_2\geqslant j+4$. Then we have
\begin{align*}
\Psi_j\psi_{j+2}v_T&=\Psi_j\psi_{j+2}\Psichaind{j-2}{b+3-2d}\Psichaind{j_2}{j+6}\Psi_{j+4}\Psi_{j+2}\Psichaind{j}{b+5-2d}v_{T_{(3)}}\\
&=\Psichaind{j_2}{j+6}\Psi_j\Psichaind{j-2}{b+3-2d}(\psi_{j+2}\Psi_{j+4}\Psi_{j+2})\Psichaind{j}{b+5-2d}v_{T_{(3)}}\\
&=\Psichaind{j_2}{j+6}\Psi_j\Psichaind{j-2}{b+3-2d}(\psi_{j+2}+\Psi_{j+4}\Psi_{j+2}\psi_{j+4})\Psichaind{j}{b+5-2d}v_{T_{(3)}}\text{ by }(*)\\
&=\Psichaind{j_2}{j+6}\Psi_j\psi_{j+2}\Psichaind{j-2}{b+3-2d}\Psichaind{j}{b+5-2d}v_{T_{(3)}}+0\\
& \ \quad\text{by $(A_s)$ for some $s<r$, as $\Psichaind{j}{b+5-2d}v_{T_{(3)}}\in\mathscr{D}$,}\\
&=0\text{ by $(A_{s'})$ for some $s'<r$, as $\Psichaind{j-2}{b+3-2d}\Psichaind{j}{b+5-2d}v_{T_{(3)}}\in\mathscr{D}$.}
\end{align*}
\item Suppose $j_2=j+2$. Then we have
\begin{align*}
\Psi_j\psi_{j+2}v_T&=\Psi_j\psi_{j+2}\Psichaind{j-2}{b+3-2d}\Psi_{j+2}\Psichaind{j}{b+5-2d}v_{T_{(3)}}\\
&=\Psi_j\Psichaind{j-2}{b+3-2d}\psi_{j+2}\Psi_{j+2}\Psichaind{j}{b+5-2d}v_{T_{(3)}}\\
&=-2\Psi_j\Psichaind{j-2}{b+3-2d}\psi_{j+2}\Psichaind{j}{b+5-2d}v_{T_{(3)}}\text{ by part 1,}\\
&=-2\Psi_j\psi_{j+2}\Psichaind{j-2}{b+3-2d}\Psichaind{j}{b+5-2d}v_{T_{(3)}}\\
&=0\text{ by $(A_s)$ for some $s<r$, as $\Psichaind{j-2}{b+3-2d}\Psichaind{j}{b+5-2d}v_{T_{(3)}}\in\mathscr{D}$.}
\end{align*}
\item Suppose $j_2=j$. Then we have
\begin{align*}
\Psi_j\psi_{j+2}v_T&=\Psi_j\psi_{j+2}\Psichaind{j-2}{b+3-2d}\Psi_j\Psichaind{j-2}{b+5-2d}v_{T_{(3)}}\\
&=\underbrace{\psi_j\psi_{j+1}\psi_{j-1}\psi_{j+2}\psi_{j-2}}_{=:\psi_*}(\psi_j\psi_{j-1}\psi_je(s_{j-2}\cdot s_j\cdot i_\lambda))\cdot\\
& \ \quad\psi_{j-3}\psi_{j-2}\psi_{j+1}\psi_{j-1}\psi_j\underbrace{\Psichaind{j-4}{b+3-2d}\Psichaind{j-2}{b+5-2d}v_{T_{(3)}}}_{=:v_{T''}\in\mathscr{D}}\\
&=\psi_*(\psi_{j-1}\psi_j\psi_{j-1}-y_{j-1}+2y_j-y_{j+1})\psi_{j-3}\psi_{j-2}\psi_{j+1}\psi_{j-1}\psi_jv_{T''}\\
&=\psi_*\psi_{j-1}\psi_j\psi_{j-3}(\psi_{j-1}\psi_{j-2}\psi_{j-1}e(s_j\cdot i_\lambda))\psi_{j+1}\psi_jv_{T''}\\
&\quad-\psi_*\psi_{j-3}(y_{j-1}\psi_{j-2}e(s_j\cdot i_\lambda))\psi_{j+1}\psi_{j-1}\psi_jv_{T''}\\
&\quad+2\psi_*\psi_{j-3}\psi_{j-2}\psi_{j+1}(y_j\psi_{j-1}e(s_j\cdot i_\lambda))\psi_jv_{T''}\\
&\quad-\psi_*\psi_{j-3}\psi_{j-2}(y_{j+1}\psi_{j+1}e(s_j\cdot i_\lambda))\psi_{j-1}\psi_jv_{T''}\\
&=\psi_*\psi_{j-1}\psi_j\psi_{j-3}\psi_{j-2}\psi_{j-1}\psi_{j-2}\psi_{j+1}\psi_jv_{T''}\\
&\quad-\psi_*\psi_{j-3}\psi_{j-2}y_{j-2}\psi_{j+1}\psi_{j-1}\psi_jv_{T''}\\
&\quad+2\psi_*\psi_{j-3}\psi_{j-2}\psi_{j+1}(\psi_{j-1}y_{j-1}+1)\psi_jv_{T''}\\
&\quad-\psi_*\psi_{j-3}\psi_{j-2}(\psi_{j+1}y_{j+2}-1)\psi_{j-1}\psi_jv_{T''}\\
&=\psi_*\psi_{j-1}\psi_{j-3}\psi_{j-2}\underbrace{\Psi_j\psi_{j-2}v_{T''}}_{=0\text{ by }(B_{r-2})}-0+0\\
&\quad+2\underbrace{\psi_j\psi_{j+1}\psi_{j-1}\psi_{j+2}}_{=:\psi^*}\psi_{j-2}\psi_{j-3}\psi_{j-2}\psi_{j+1}\psi_jv_{T''}\\
&\quad-0+\psi^*\psi_{j-2}\psi_{j-3}\psi_{j-2}\psi_{j-1}\psi_jv_{T''}\\
&=2\psi^*(\psi_{j-2}\psi_{j-3}\psi_{j-2}e(s_j\cdot i_\lambda))\psi_{j+1}\psi_jv_{T''}\\
&\quad+\psi^*(\psi_{j-2}\psi_{j-3}\psi_{j-2}e(s_j\cdot i_\lambda))\psi_{j-1}\psi_jv_{T''}\\
&=2\psi^*(\psi_{j-3}\psi_{j-2}\psi_{j-3}-y_{j-3}+2y_{j-2}-y_{j-1})\psi_{j+1}\psi_jv_{T''}\\
&\quad+\psi^*(\psi_{j-3}\psi_{j-2}\psi_{j-3}-y_{j-3}+2y_{j-2}-y_{j-1})\psi_{j-1}\psi_jv_{T''}\\
&=2\psi^*\psi_{j-3}\psi_{j-2}\psi_{j+1}\psi_j\underbrace{\psi_{j-3}v_{T''}}_{=0\text{ as $j{-}3$ is even}}-0+0-0\\
&\quad+\psi^*\psi_{j-3}\psi_{j-2}\psi_{j-1}\psi_j\underbrace{\psi_{j-3}v_{T''}}_{=0}-0+0\\
&\quad-\psi^*(y_{j-1}\psi_{j-1}e(s_j\cdot i_\lambda))\psi_jv_{T''}\\
&=-\psi^*(\psi_{j-1}y_j-1)\psi_jv_{T''}\\
&=-\psi^*\psi_{j-1}\psi_jy_{j+1}v_{T''}+\psi^*\psi_jv_{T''}\\
&=-0+\underbrace{\Psi_j\psi_{j+2}v_{T''}.}_{=0\text{ by }(A_{r-2})}\\
\end{align*}
\end{enumerate}
\end{enumerate}
Next, we show that $(B_r)$ follows if $(A_s)$ and $(B_s)$ hold for all $s<r$.

For $(B_0)$, we have $\Psi_j\psi_{j-2}z=\psi_j\psi_{j+1}\psi_{j-1}\psi_{j-2}\psi_jz=0$ as at least one of $\psi_j$, $\psi_{j-2}$ must annihilate $z$.

Now suppose $r>0$.

If $j_1\geqslant j+4$ or $j_1\leqslant j-6$, then we clearly have $\Psi_j\psi_{j-2}v_T=\Psi_{j_1}\Psi_j\psi_{j-2}v_{T'}$ and our result follows from $(B_{r-1})$. Once again, we break the proof up for the remaining four possible values of $j_1$.
\begin{enumerate}
\item Suppose $j_1=j+2$. If $b+3-2d=j+2$ then we have
\begin{align*}
\Psi_j\psi_{j-2}v_T&=\Psi_j\psi_{j-2}\Psi_{j+2}v_{T_{(2)}}\\
&=0
\end{align*}
as $\psi_{j-2}$ commutes with everything to its right, since the lowest indexed term in $v_{T_{(2)}}$ is $\Psi_{j+4}$.

If $b+3-2d\leqslant j$, we have
\begin{align*}
\Psi_j\psi_{j-2}v_T&=\Psi_j\psi_{j-2}\Psi_{j+2}\Psi_j\Psichaind{j-2}{b+3-2d}v_{T_{(2)}}\\
&=\underbrace{\psi_j\psi_{j+1}\psi_{j-1}\psi_{j-2}\psi_{j+2}\psi_{j+3}}_{=:\psi_*}(\psi_j\psi_{j+1}\psi_je(s_{j+2}\cdot s_j\cdot i_\lambda))\cdot\\
&\ \quad\psi_{j+2}\psi_{j+1}\psi_{j-1}\psi_j\Psichaind{j-2}{b+3-2d}v_{T_{(2)}}\\
&=\psi_*(\psi_{j+1}\psi_j\psi_{j+1}+y_j-2y_{j+1}+y_{j+2})\psi_{j+2}\psi_{j+1}\psi_{j-1}\psi_j\Psichaind{j-2}{b+3-2d}v_{T_{(2)}}\\
&=\psi_*\psi_{j+1}\psi_j(\psi_{j+1}\psi_{j+2}\psi_{j+1}e(s_j\cdot i_\lambda))\psi_{j-1}\psi_j\Psichaind{j-2}{b+3-2d}v_{T_{(2)}}\\
&\quad+\psi_*\psi_{j+2}\psi_{j+1}(y_j\psi_{j-1}e(s_j \cdot i_\lambda))\psi_j\Psichaind{j-2}{b+3-2d}v_{T_{(2)}}\\
&\quad-2\psi_*\psi_{j+2}(y_{j+1}\psi_{j+1}e(s_j\cdot i_\lambda))\psi_{j-1}\psi_j\Psichaind{j-2}{b+3-2d}v_{T_{(2)}}\\
&\quad+\psi_*(y_{j+2}\psi_{j+2}e(s_j\cdot i_\lambda))\psi_{j+1}\psi_{j-1}\psi_j\Psichaind{j-2}{b+3-2d}v_{T_{(2)}}\\
&=\psi_*\psi_{j+1}\psi_j\psi_{j+2}\psi_{j+1}\psi_{j+2}\psi_{j-1}\psi_j\Psichaind{j-2}{b+3-2d}v_{T_{(2)}}\\
&\quad+\psi_*\psi_{j+2}\psi_{j+1}(\psi_{j-1}y_{j-1}+1)\psi_j\Psichaind{j-2}{b+3-2d}v_{T_{(2)}}\\
&\quad-2\psi_*\psi_{j+2}(\psi_{j+1}y_{j+2}-1)\psi_{j-1}\psi_j\Psichaind{j-2}{b+3-2d}v_{T_{(2)}}\\
&\quad+\psi_*\psi_{j+2}y_{j+3}\psi_{j+1}\psi_{j-1}\psi_j\Psichaind{j-2}{b+3-2d}v_{T_{(2)}}\\
&=\psi_*\psi_{j+1}\psi_{j+2}\underbrace{\Psi_j\psi_{j+2}\Psichaind{j-2}{b+3-2d}v_{T_{(2)}}}_{=0\text{ by $(A_{r-2})$}}\\
&\quad+0+\psi_j\psi_{j+1}\psi_{j-1}\psi_{j-2}(\psi_{j+2}\psi_{j+3}\psi_{j+2}e(s_j\cdot i_\lambda))\psi_{j+1}\psi_j\Psichaind{j-2}{b+3-2d}v_{T_{(2)}}\\
&\quad-0+2\psi_j\psi_{j+1}\psi_{j-1}\psi_{j-2}(\psi_{j+2}\psi_{j+3}\psi_{j+2}e(s_j\cdot i_\lambda))\psi_{j-1}\psi_j\Psichaind{j-2}{b+3-2d}v_{T_{(2)}}\\
&\quad+0\\
&=(\psi_j\psi_{j+1}\psi_{j-1}\psi_{j-2}(\psi_{j+3}\psi_{j+2}\psi_{j+3}+y_{j+2}-2y_{j+3}+y_{j+4})\psi_{j+1}\\
&\quad+2\psi_j\psi_{j+1}\psi_{j-1}\psi_{j-2}(\psi_{j+3}\psi_{j+2}\psi_{j+3}+y_{j+2}-2y_{j+3}+y_{j+4})\psi_{j-1})\cdot\\
&\ \quad\psi_j\Psichaind{j-2}{b+3-2d}v_{T_{(2)}}\\
&=0+\psi_j\psi_{j+1}\psi_{j-1}\psi_{j-2}(y_{j+2}\psi_{j+1}e(s_j\cdot i_\lambda))\psi_j\Psichaind{j-2}{b+3-2d}v_{T_{(2)}}-0+0\\
&\quad+0+0-0+0\\
&=\psi_j\psi_{j+1}\psi_{j-1}\psi_{j-2}(\psi_{j+1}y_{j+1}+1)\psi_j\Psichaind{j-2}{b+3-2d}v_{T_{(2)}}\\
&=0+\Psi_j\psi_{j-2}\Psichaind{j-2}{b+3-2d}v_{T_{(2)}}\\
&=0\text{ by $(B_{r-2})$.}
\end{align*}
\item Suppose $j_1=j$. Then we have
\begin{align*}
\Psi_j\psi_{j-2}v_T&=-2\Psi_j\psi_{j-2}v_{T'}\text{ by part 1,}\\
&=0\text{ by }(B_{s-1}).
\end{align*}
\item Suppose $j_1=j-2$. Then we have
\begin{align*}
\Psi_j\psi_{j-2}v_T&=-2\Psi_j\psi_{j-2}v_{T'}\text{ by part 1,}\\
&=0\text{ by }(B_{r-1}).
\end{align*}
\item Suppose $j_1=j-4$. We divide into subcases.
\begin{enumerate}
\item Suppose $j_2\geqslant j+2$. Then we have
\begin{align*}
\Psi_j\psi_{j-2}v_T&=\Psi_j\psi_{j-2}\Psi_{j-4}\Psichaind{j-6}{b+3-2d}\Psichaind{j_2}{j+4}\Psi_{j+2}\Psi_j\Psichaind{j-2}{b+5-2d}v_{T_{(3)}}\\
&=\psi_j\psi_{j+1}\psi_{j-1}\psi_{j-2}\Psi_{j-4}\Psichaind{j-6}{b+3-2d}\Psichaind{j_2}{j+4}(\psi_j\Psi_{j+2}\Psi_j)\Psichaind{j-2}{b+5-2d}v_{T_{(3)}}\\
&=\psi_j\psi_{j+1}\psi_{j-1}\psi_{j-2}\Psi_{j-4}\Psichaind{j-6}{b+3-2d}\Psichaind{j_2}{j+4}(\psi_j+\Psi_{j+2}\Psi_j\psi_{j+2})\cdot\\
&\ \quad\Psichaind{j-2}{b+5-2d}v_{T_{(3)}}\text{ by $(*)$,}\\
&=\Psi_j\psi_{j-2}\Psichaind{j-4}{b+3-2d}\Psichaind{j_2}{j+4}\Psichaind{j-2}{b+5-2d}v_{T_{(3)}}\\
&\quad+0\text{ by $(A_s)$ for some $s<r$, as $\Psichaind{j-2}{b+5-2d}v_{T_{(3)}}\in\mathscr{D}$,}\\
&=0\text{ by $(B_{s'})$ for some $s'<r$, as $\Psichaind{j-4}{b+3-2d}\Psichaind{j_2}{j+4}\Psichaind{j-2}{b+5-2d}v_{T_{(3)}}\in\mathscr{D}$.}
\end{align*}
\item Suppose $j_2=j$. Then we have
\begin{align*}
\Psi_j\psi_{j-2}v_T&=\Psi_j\psi_{j-2}\Psichaind{j-4}{b+3-2d}\Psi_j\Psichaind{j-2}{b+5-2d}v_{T_{(3)}}\\
&=-2\Psi_j\psi_{j-2}\Psichaind{j-4}{b+3-2d}\Psichaind{j-2}{b+5-2d}v_{T_{(3)}}\text{ by part 1,}\\
&=0\text{ by $(B_{r-1})$, as $\Psichaind{j-4}{b+3-2d}\Psichaind{j-2}{b+5-2d}v_{T_{(3)}}\in\mathscr{D}$.}
\end{align*}
\item Suppose $j_2=j-2$. Then we have
\begin{align*}
\Psi_j\psi_{j-2}v_T&=\Psi_j\psi_{j-2}\Psi_{j-4}\Psi_{j-2}\underbrace{\Psichaind{j-6}{b+3-2d}\Psichaind{j-4}{b+5-2d}v_{T_{(3)}}}_{=:v_{T''}}\\
&=\Psi_j\psi_{j-4}\psi_{j-5}(\psi_{j-2}\psi_{j-3}\psi_{j-2}e(s_{j-4}\cdot s_{j-2}\cdot i_\lambda))\psi_{j-4}\psi_{j-1}\psi_{j-3}\psi_{j-2}v_{T''}\\
&=\Psi_j\psi_{j-4}\psi_{j-5}(\psi_{j-3}\psi_{j-2}\psi_{j-3}-y_{j-3}+2y_{j-2}-y_{j-1})\cdot\\
&\quad\psi_{j-4}\psi_{j-1}\psi_{j-3}\psi_{j-2}v_{T''}\\
&=\Psi_j\psi_{j-4}\psi_{j-5}\psi_{j-3}\psi_{j-2}(\psi_{j-3}\psi_{j-4}\psi_{j-3}e(s_{j-2}\cdot i_\lambda))\psi_{j-1}\psi_{j-2}v_{T''}\\
&\quad-\Psi_j\psi_{j-4}\psi_{j-5}(y_{j-3}\psi_{j-4}e(s_{j-2}\cdot i_\lambda))\psi_{j-1}\psi_{j-3}\psi_{j-2}v_{T''}\\
&\quad+2\Psi_j\psi_{j-4}\psi_{j-5}\psi_{j-4}\psi_{j-1}(y_{j-2}\psi_{j-3}e(s_{j-2}\cdot i_\lambda))\psi_{j-2}v_{T''}\\
&\quad-\Psi_j\psi_{j-4}\psi_{j-5}\psi_{j-4}(y_{j-1}\psi_{j-1}e(s_{j-2}\cdot i_\lambda))\psi_{j-3}\psi_{j-2}v_{T''}\\
&=\Psi_j\psi_{j-4}\psi_{j-5}\psi_{j-3}\psi_{j-2}\psi_{j-4}\psi_{j-3}\psi_{j-4}\psi_{j-1}\psi_{j-2}v_{T''}\\
&\quad-\Psi_j\psi_{j-4}\psi_{j-5}\psi_{j-4}y_{j-4}\psi_{j-1}\psi_{j-3}\psi_{j-2}v_{T''}\\
&\quad+2\Psi_j\psi_{j-4}\psi_{j-5}\psi_{j-4}\psi_{j-1}(\psi_{j-3}y_{j-3}+1)\psi_{j-2}v_{T''}\\
&\quad-\Psi_j\psi_{j-4}\psi_{j-5}\psi_{j-4}(\psi_{j-1}y_j-1)\psi_{j-3}\psi_{j-2}v_{T''}\\
&=\Psi_j\psi_{j-4}\psi_{j-5}\psi_{j-3}\psi_{j-4}\underbrace{\Psi_{j-2}\psi_{j-4}v_{T''}}_{=0\text{ by $(B_{r-2})$}}-0\\
&\quad+0+2\Psi_j(\psi_{j-4}\psi_{j-5}\psi_{j-4}e(s_{j-2}\cdot i_\lambda))\psi_{j-1}\psi_{j-2}v_{T''}\\
&\quad-0+\Psi_j(\psi_{j-4}\psi_{j-5}\psi_{j-4}e(s_{j-2}\cdot i_\lambda))\psi_{j-3}\psi_{j-2}v_{T''}\\
&=2\Psi_j(\psi_{j-5}\psi_{j-4}\psi_{j-5}-y_{j-5}+2y_{j-4}-y_{j-3})\psi_{j-1}\psi_{j-2}v_{T''}\\
&\quad+\Psi_j(\psi_{j-5}\psi_{j-4}\psi_{j-5}-y_{j-5}+2y_{j-4}-y_{j-3})\psi_{j-3}\psi_{j-2}v_{T''}\\
&=0-0+0-0+0-0+0-\Psi_j(y_{j-3}\psi_{j-3}e(s_{j-2}\cdot i_\lambda))\psi_{j-2}v_{T''}\\
&=-\Psi_j(\psi_{j-3}y_{j-2}-1)\psi_{j-2}v_{T''}\\
&=0+\Psi_j\psi_{j-2}v_{T''}\\
&=0\text{ by $(B_{r-2})$}.
\end{align*}
Note that both inductive steps are possible because $v_{T''}\in\mathscr{D}$.
\end{enumerate}
\end{enumerate}
This completes our proof of statements 2--5.\qedhere
\end{enumerate}
\end{proof}

In the next two results, we are concerned with how the Garnir element $\psi_1\psi_2\dots\psi_{b+1}$ acts on elements of $\mathscr{D}$.

\begin{lem}
Suppose $j$ is odd with $3\leqslant j\leqslant n-2$, and $T\in\Dom(\lambda)$. Then
\begin{enumerate}
\item For all odd $n-2\geqslant i\geqslant j+4$, $\psi_1\psi_2\dots\psi_j\Psi_iv_T=\Psi_i\psi_1\psi_2\dots\psi_jv_T$.
\item $\psi_1\psi_2\dots\psi_j\Psi_{j+2}v_T=\psi_{j+2}\psi_{j+3}\psi_1\psi_2\dots\psi_{j+2}v_T$.
\item $\psi_1\psi_2\dots\psi_j\Psi_jv_T=-2\psi_1\psi_2\dots\psi_jv_T$.
\item $\psi_1\psi_2\dots\psi_j\Psi_{j-2}v_T=\Psi_{j-1}\psi_1\psi_2\dots\psi_jv_T+\psi_j\psi_{j-1}\psi_1\psi_2\dots\psi_{j-2}v_T$.
\item For all odd $3\leqslant i\leqslant j-4$, \[\psi_1\psi_2\dots\psi_j\Psi_iv_T=\Psi_{i+1}\psi_1\psi_2\dots\psi_jv_T+\psi_{i+2}\psi_{i+1}\psi_{i+3}\psi_{i+4}\dots\psi_j\psi_1\psi_2\dots\psi_iv_T.\]
\end{enumerate}
\end{lem}

\begin{proof}
1 and 2 follow immediately from definitions and the commuting relations between $\psi$ generators. 3 follows immediately from Lemma \ref{psi-actions}. So only statements 4 and 5 require any real work!
\begin{enumerate}
{\item[4.]\abovedisplayskip=0pt\abovedisplayshortskip=0pt~\vspace*{-\baselineskip}
\begin{align*}
\psi_1\psi_2\dots\psi_j\Psi_{j-2}v_T&=\underbrace{\psi_1\psi_2\dots\psi_{j-4}}_{=:\psi^*}\psi_{j-3}(\psi_{j-2}\psi_{j-1}\psi_{j-2}e(s_j\cdot s_{j-2}\cdot i_\lambda))\psi_j\psi_{j-1}\psi_{j-3}\psi_{j-2}v_T\\
&=\psi^*\psi_{j-3}(\psi_{j-1}\psi_{j-2}\psi_{j-1}+y_{j-2}-2y_{j-1}+y_j)\psi_j\psi_{j-1}\psi_{j-3}\psi_{j-2}v_T\\
&=\psi_{j-1}\psi^*\psi_{j-3}\psi_{j-2}(\psi_{j-1}\psi_j\psi_{j-1}e(s_{j-2}\cdot i_\lambda))\psi_{j-3}\psi_{j-2}v_T\\
&\quad+\psi_j\psi_{j-1}\psi^*\psi_{j-3}(y_{j-2}\psi_{j-3}e(s_{j-2}\cdot i_\lambda))\psi_{j-2}v_T\\
&\quad-2\psi_j\psi^*\psi_{j-3}(y_{j-1}\psi_{j-1}e(s_{j-2}\cdot i_\lambda))\psi_{j-3}\psi_{j-2}v_T\\
&\quad+\psi^*\psi_{j-3}(y_j\psi_je(s_{j-2}\cdot i_\lambda))\psi_{j-1}\psi_{j-3}\psi_{j-2}v_T\\
&=\psi_{j-1}\psi^*\psi_{j-3}\psi_{j-2}(\psi_j\psi_{j-1}\psi_j)\psi_{j-3}\psi_{j-2}v_T\\
&\quad+\psi_j\psi_{j-1}\psi^*\psi_{j-3}(\psi_{j-3}y_{j-3}+1)\psi_{j-2}v_T\\
&\quad-2\psi_j\psi^*\psi_{j-3}(\psi_{j-1}y_j-1)\psi_{j-3}\psi_{j-2}v_T\\
&\quad+\psi^*\psi_{j-3}(\psi_jy_{j+1})\psi_{j-1}\psi_{j-3}\psi_{j-2}v_T\\
&=\psi_{j-1}\psi_j\psi^*(\psi_{j-3}\psi_{j-2}\psi_{j-3}e(s_{j-1}\cdot s_j\cdot s_{j-2}\cdot i_\lambda))\psi_{j-1}\psi_j\psi_{j-2}v_T\\
&\quad+0+\psi_j\psi_{j-1}\psi^*\psi_{j-3}\psi_{j-2}v_T\\
&\quad-0+2\psi_j\psi^*(\psi_{j-3}^2e(s_{j-2}\cdot i_\lambda))\psi_{j-2}v_T+0\\
&=\psi_{j-1}\psi_j\psi^*(\psi_{j-2}\psi_{j-3}\psi_{j-2})\psi_{j-1}\psi_j\psi_{j-2}v_T\\
&\quad+\psi_j\psi_{j-1}\psi^*\psi_{j-3}\psi_{j-2}v_T+0\\
&=\psi_{j-1}\psi_j\psi_{j-2}\psi^*\psi_{j-3}(\psi_{j-2}\psi_{j-1}\psi_{j-2}e(s_j\cdot i_\lambda))\psi_jv_T\\
&\quad+\psi_j\psi_{j-1}\psi^*\psi_{j-3}\psi_{j-2}v_T\\
&=\psi_{j-1}\psi_j\psi_{j-2}\psi^*\psi_{j-3}(\psi_{j-1}\psi_{j-2}\psi_{j-1})\psi_jv_T+\psi_j\psi_{j-1}\psi^*\psi_{j-3}\psi_{j-2}v_T\\
&=\Psi_{j-1}\psi^*\psi_{j-3}\psi_{j-2}\psi_{j-1}\psi_jv_T+\psi_j\psi_{j-1}\psi^*\psi_{j-3}\psi_{j-2}v_T.
\end{align*}}
\item [5.] Let $i$ be odd and $4\leqslant i\leqslant j-4$. Then
\begin{align*}
\psi_1\psi_2\dots\psi_j\Psi_iv_T&=\underbrace{\psi_1\psi_2\dots\psi_{i-2}}_{\psi_*}\psi_{i-1}\psi_i\psi_{i+1}\psi_{i+2}\Psi_i\psi_{i+3}\underbrace{\psi_{i+4}\dots\psi_j}_{\psi^*}v_T\\
&=\psi_*\psi_{i-1}(\psi_i\psi_{i+1}\psi_ie(s_{i+2}\cdot s_i\cdot s_{i+4}\cdot s_{i+6}\cdots s_j\cdot i_\lambda))\cdot\\
&\ \quad\psi_{i+2}\psi_{i+1}\psi_{i-1}\psi_i\psi_{i+3}\psi^*v_T\\
&=\psi_*\psi_{i-1}(\psi_{i+1}\psi_i\psi_{i+1}+y_i-2y_{i+1}+y_{i+2})\psi_{i+2}\psi_{i+1}\psi_{i-1}\psi_i\psi_{i+3}\psi^*v_T\\
&=\psi_{i+1}\psi_*\psi_{i-1}\psi_i(\psi_{i+1}\psi_{i+2}\psi_{i+1}e(s_i\cdot s_{i+4}\cdot s_{i+6}\cdots s_j\cdot i_\lambda))\cdot\\
&\ \quad\psi_{i-1}\psi_i\psi_{i+3}\psi^*v_T\\
&\quad+\psi_{i+2}\psi_{i+1}\psi_*\psi_{i-1}(y_i\psi_{i-1}e(s_i\cdot s_{i+4}\cdot s_{i+6}\cdots s_j\cdot i_\lambda))\psi_i\psi_{i+3}\psi^*v_T\\
&\quad-2\psi_{i+2}\psi_*\psi_{i-1}(y_{i+1}\psi_{i+1}e(s_i\cdot s_{i+4}\cdot s_{i+6}\cdots s_j\cdot i_\lambda))\psi_{i-1}\psi_i\psi_{i+3}\psi^*v_T\\
&\quad+\psi_*\psi_{i-1}(y_{i+2}\psi_{i+2}e(s_i\cdot s_{i+4}\cdot s_{i+6}\cdots s_j\cdot i_\lambda))\psi_{i+1}\psi_{i-1}\psi_i\psi_{i+3}\psi^*v_T\\
&=\psi_{i+1}\psi_*\psi_{i-1}\psi_i(\psi_{i+2}\psi_{i+1}\psi_{i+2})\psi_{i-1}\psi_i\psi_{i+3}\psi^*v_T\\
&\quad+\psi_{i+2}\psi_{i+1}\psi_*\psi_{i-1}(\psi_{i-1}y_{i-1}+1)\psi_i\psi_{i+3}\psi^*v_T\\
&\quad-2\psi_{i+2}\psi_*\psi_{i-1}(\psi_{i+1}y_{i+2}-1)\psi_{i-1}\psi_i\psi_{i+3}\psi^*v_T\\
&\quad+\psi_*\psi_{i-1}(\psi_{i+2}y_{i+3})\psi_{i+1}\psi_{i-1}\psi_i\psi_{i+3}\psi^*v_T\\
&=\psi_{i+1}\psi_{i+2}\psi_*(\psi_{i-1}\psi_i\psi_{i-1}e(s_{i+1}\cdot s_{i+2}\cdot s_i\cdot s_{i+4}\cdots s_j\cdot i_\lambda))\cdot\\
&\ \quad\psi_{i+1}\psi_{i+2}\psi_i\psi_{i+3}\psi^*v_T+0+\psi_{i+2}\psi_{i+1}\psi_*\psi_{i-1}\psi_i\psi_{i+3}\psi^*v_T\\
&\quad-0+2\psi_{i+2}\psi_*(\psi_{i-1}^2e(s_i\cdot s_{i+4}\cdot s_{i+6}\cdots s_j\cdot i_\lambda))\psi_i\psi_{i+3}\psi^*v_T\\
&\quad+\psi_{i+2}\psi_{i+1}\psi_*(\psi_{i-1}^2e(s_i\cdot s_{i+4}\cdot s_{i+6}\cdots s_j\cdot i_\lambda))\psi_iy_{i+3}\psi_{i+3}\psi^*v_T\\
&=\psi_{i+1}\psi_{i+2}\psi_*(\psi_i\psi_{i-1}\psi_i)\psi_{i+1}\psi_{i+2}\psi_i\psi_{i+3}\psi^*v_T\\
&\quad+\psi_{i+2}\psi_{i+1}\psi_{i+3}\psi^*\psi_*\psi_{i-1}\psi_iv_T+0+0\\
&=\psi_{i+1}\psi_{i+2}\psi_i\psi_*\psi_{i-1}(\psi_i\psi_{i+1}\psi_ie(s_{i+2}\cdot s_{i+4}\cdots s_j\cdot i_\lambda))\psi_{i+2}\psi_{i+3}\psi^*v_T\\
&\quad+\psi_{i+2}\psi_{i+1}\psi_{i+3}\psi^*\psi_*\psi_{i-1}\psi_iv_T\\
&=\psi_{i+1}\psi_{i+2}\psi_i\psi_*\psi_{i-1}(\psi_{i+1}\psi_i\psi_{i+1})\psi_{i+2}\psi_{i+3}\psi^*v_T\\
&\quad+\psi_{i+2}\psi_{i+1}\psi_{i+3}\psi^*\psi_*\psi_{i-1}\psi_iv_T\\
&=\Psi_{i+1}\psi_*\psi_{i-1}\psi_i\psi_{i+1}\psi_{i+2}\psi_{i+3}\psi^*v_T+\psi_{i+2}\psi_{i+1}\psi_{i+3}\psi^*\psi_*\psi_{i-1}\psi_iv_T.\\
& & \qedhere
\end{align*}
\end{enumerate}
\end{proof}

\begin{prop}
Let $T\in\Dom(\lambda)$. Then $\psi_1\psi_2\dots\psi_{b+1}v_T=0$.
\end{prop}

\begin{proof}
Repeated application of the above lemma yields $\psi_1\psi_2\dots\psi_{b+1}v_T$ as a sum of expressions ending in $\psi_1\psi_2\dots\psi_jz$ for various odd values of $j\geqslant 3$. In all cases the relations of the Specht module give us our result.
\end{proof}

\section{Decomposability when $n$ is odd}

We can now begin calculating $\mathscr{H}$-endomorphisms of $S_\lambda$. We now know that $f\in\End_\mathscr{H}(S_\lambda)$ if and only if \[f(z)=\sum_{T\in\Dom(\lambda)}\alpha_Tv_T\text{ for some }\alpha_T\in\mathbb{F}\] with $\psi_jf(z)=0$ for all odd $j\neq b+1$ with $3\leqslant j\leqslant n-2$.

\begin{defn}
Let $i,j$ be odd integers with $3\leqslant i\leqslant b+1<j\leqslant n$. We will denote by $T_{i,j}$ the tableau with dominoes $\{[2,3],[4,5],\dots,[b,b+1],[j-1,j]\}\setminus\{[i-1,i]\}$ in the leg.
\end{defn}

\begin{eg}
If $\lambda=(5,1^4)$ then $T_{5,9}=\young(14567,2,3,8,9)$ and $T_{3,7}=\young(12389,4,5,6,7)$.
\end{eg}

\begin{rem}
We observe that the normal form for $v_{T_{i,j}}$ is $\Psichainu{i}{b-1}\Psichaind{j-2}{b+1}z$.
\end{rem}

\begin{prop}\label{endomorphismf}
Suppose $a$ is odd and $b$ is even. Then there exists an $\mathscr{H}$-endomorphism $f$ of $S_\lambda$ given by \[f(z)=\sum_{\substack{3\leqslant i\leqslant b+1\\ b+3\leqslant j\leqslant n\\i,j \text{ odd}}}\mfrac{i-1}{2}\cdot\mfrac{n+2-j}{2}v_{T_{i,j}}.\]
\end{prop}

\begin{proof}
All we need to show is that $\psi_kf(z)=0$ for all odd $k\neq b+1$ with $3\leqslant k\leqslant n-2$. We will rely extensively on our previous results regarding the actions of $\psi$ generators on tableaux.

First, notice that $\psi_3v_{T_{i,j}}=0$ for all $i\geqslant 7$. So
\begin{align*}
\psi_3f(z)&=\psi_3\left(\sum_j2\cdot\mfrac{n+2-j}{2}v_{T_{5,j}}+\mfrac{n+2-j}{2}v_{T_{3,j}}\right)\\
&=\sum_j\mfrac{n+2-j}{2}(2\psi_3\cdot v_{T_{5,j}}-2\psi_3\cdot v_{T_{5,j}})\\
&=0.
\end{align*}

Next, suppose $5\leqslant k\leqslant b-1$. We notice that $\psi_kv_{T_{i,j}}=0$ for all $i\leqslant k-4$ and for all $i\geqslant k+4$. So
\begin{align*}
\psi_kf(z)&=\psi_k\left(\sum_j\mfrac{k+1}{2}\cdot\mfrac{n+2-j}{2}v_{T_{k+2,j}}+\mfrac{k-1}{2}\cdot\mfrac{n+2-j}{2}v_{T_{k,j}}\right.\\
&\quad\left.+\mfrac{k-3}{2}\cdot\mfrac{n+2-j}{2}v_{T_{k-2,j}}\right)\\
&=\sum_j\mfrac{n+2-j}{2}\left(\mfrac{k+1}{2}-2\cdot\mfrac{k-1}{2}+\mfrac{k-3}{2}\right)\psi_kv_{T_{k+2,j}}\\
&=0.
\end{align*}

Now, for $b+3\leqslant k\leqslant n-4$, we notice that $\psi_kv_{T_{i,j}}=0$ for all $j\leqslant k-2$ and for all $j\geqslant k+6$. So
\begin{align*}
\psi_kf(z)&=\psi_k\left(\sum_i\mfrac{i-1}{2}\cdot\mfrac{n+2-k}{2}v_{T_{i,k}}+\mfrac{i-1}{2}\cdot\mfrac{n-k}{2}v_{T_{i,k+2}}\right.\\
&\quad+\left.\mfrac{i-1}{2}\cdot\mfrac{n-k-2}{2}v_{T_{i,k+4}}\right)\\
&=\sum_i\mfrac{i-1}{2}\left(\mfrac{n+2-k}{2}-2\cdot\mfrac{n-k}{2}+\mfrac{n-k-2}{2}\right)\psi_kv_{T_{i,k}}\\
&=0.
\end{align*}
Finally, we notice that $\psi_{n-2}v_{T_{i,j}}=0$ unless $j=n-2$ or $n$. So
\begin{align*}
\psi_{n-2}f(z)&=\psi_{n-2}\left(\sum_i\mfrac{i-1}{2}\cdot 2\cdot v_{T_{i,n-2}}+\mfrac{i-1}{2}v_{T_{i,n}}\right)\\
&=\sum_i(i-1)\psi_{n-2}v_{T_{i,n-2}}-2\cdot\mfrac{i-1}{2}\psi_{n-2}v_{T_{i,n-2}}\\
&=0.\qedhere
\end{align*}
\end{proof}

\begin{rem}
This endomorphism allows us to tackle our decomposability question. In particular, $S_\lambda$ can be decomposed into a direct sum of the generalised eigenspaces of $f$. That is $E_x=\{v\in S_\lambda \ | (f-xI)^nv=0\text{ for some }n\in\mathbb{N}\}$ for each eigenvalue $x$ of $f$, and \[S_\lambda=\bigoplus_{\text{$x$ an eigenvalue of $f$}}E_x.\] From the definition of $E_x$ it is clear that it is a non-zero $\mathscr{H}$-module whenever $x$ is an eigenvalue of $f$. The existence of two distinct eigenvalues of $f$ would ensure that we have at least two non-trivial summands in the decomposition above, and we would be done.
\end{rem}

The following lemma will be used repeatedly in further proofs.

\begin{lem}
Suppose $x_1\geqslant y_1\geqslant 3$ and $x_2\geqslant y_2\geqslant 3$ are all odd numbers. Suppose also that $X\in e(i_\lambda)S_\lambda$. Then we have the following cancellation relations:
\begin{enumerate}
\item If $x_1\geqslant x_2\geqslant y_1$ we have \[\Psichaind{x_1}{y_1}\Psichaind{x_2}{y_2}X=\Psichaind{x_2-4}{y_1}\Psichaind{x_1}{y_2}X.\]\\
\item If $x_2\geqslant y_1\geqslant y_2$ we have \[\Psichaind{x_1}{y_1}\Psichaind{x_2}{y_2}X=\Psichaind{x_1}{y_2}\Psichaind{x_2}{y_1+4}X.\]
\end{enumerate}
\end{lem}

\begin{proof}
\begin{enumerate}
\Item \begin{align*}
\Psichaind{x_1}{y_1}\Psichaind{x_2}{y_2}X&=\Psichaind{x_1}{x_2+2}\Psi_{x_2}\Psi_{x_2-2}\Psichaind{x_2-4}{y_1}\Psichaind{x_2}{y_2}X\\
&=\Psichaind{x_1}{x_2+2}\Psi_{x_2}\Psi_{x_2-2}\Psi_{x_2}\Psichaind{x_2-4}{y_1}\Psichaind{x_2-2}{y_2}X\\
&=\Psichaind{x_1}{x_2+2}\Psi_{x_2}\Psichaind{x_2-4}{y_1}\Psichaind{x_2-2}{y_2}X\\
&=\Psichaind{x_2-4}{y_1}\Psichaind{x_1}{y_2}X.
\end{align*}
\item The proof proceeds similarly to the previous case.\qedhere
\end{enumerate}
\end{proof}

Now, we work towards computing the eigenvalues of $f$. It is clear that $f$ acts on $e(i_\lambda)S_\lambda$; $f(v_T)\in e(i_\lambda)S_\lambda$ whenever $T\in\Dom(\lambda)$ by the nature of our actions of $\psi$ generators on elements of $\mathscr{D}$. We will show that the action of $f$ on $e(i_\lambda)S_\lambda$ is triangular. Take $T\in \Dom(\lambda)$, and write $v_T$ in normal form:
\begin{align*}
v_T&=\Psichaind{j_1}{b+3-2d}\Psichaind{j_2}{b+5-2d}\dots\Psichaind{j_d}{b+1}z.\\
\intertext{Then we want to look at}
f(v_T)&=\Psichaind{j_1}{b+3-2d}\Psichaind{j_2}{b+5-2d}\dots\Psichaind{j_d}{b+1}\cdot f(z)\\
&=\sum_{\substack{3\leqslant i\leqslant b+1\\ b+3\leqslant j\leqslant n\\i,j \text{ odd}}}\mfrac{i-1}{2}\cdot\mfrac{n+2-j}{2}\Psichaind{j_1}{b+3-2d}\Psichaind{j_2}{b+5-2d}\dots\Psichaind{j_d}{b+1}\cdot\Psichainu{i}{b-1}\Psichaind{j-2}{b+1}z.
\end{align*}

We begin by looking at the simplified case where $d=1$.

\begin{lem}\label{triangularsimple} Let $3\leqslant i\leqslant b+1<j\leqslant n$  and $j_d\geqslant b+1$. Then \[\Psichaind{j_d}{b+1}\cdot\Psichainu{i}{b-1}\Psichaind{j-2}{b+1}z=\begin{cases}
-2\Psichaind{j_d}{b+1}z & \text{if $i=b+1$ and $j=b+3$,}\\
\Psichaind{j_d}{b+1}z & \text{if $i=b+1$ and $j=b+5$,}\\
0 & \text{if $i=b+1$ and $j\geqslant b+7$,}\\
\Psichaind{j_d}{b+1}z & \text{if $i=b-1$ and $j=b+3$,}\\
\Psichainu{i}{b-3}\Psichaind{j_d}{b-1}\Psichaind{j-2}{b+1}z & \text{if $i\leqslant b-1$ and $j_d\leqslant j-4$,}\\
0 & \text{if $i<b-1$ and $j_d\geqslant j-2$ and $j=b+3$,}\\
\Psichainu{i}{b-3}\Psichaind{j-6}{b-1}\Psichaind{j_d}{b+1}z & \text{if $i\leqslant b-1$ and $j_d\geqslant j-2$ and $j\geqslant b+5$.}
\end{cases}\]
\end{lem}

\begin{proof}
First suppose $i=b+1$. If $j=b+3$ we have \[\Psichaind{j_d}{b+1}\cdot\Psi_{b+1}z=-2\Psichaind{j_d}{b+1}z.\]

If $j=b+5$ we have \[\Psichaind{j_d}{b+1}\cdot\Psi_{b+3}\Psi_{b+1}z =\Psichaind{j_d}{b+1}z.\]

If $j\geqslant b+7$ we have
\begin{align*}
\Psichaind{j_d}{b+1}\cdot\Psichaind{j-2}{b+1}z&=\Psichaind{j_d}{b+1}\cdot\Psichaind{j-2}{b+5}z\\
&=0.
\end{align*}

If $i=b-1$ and $j=b+3$ we have \[\Psichaind{j_d}{b+1}\cdot\Psi_{b-1}\Psi_{b+1}z=\Psichaind{j_d}{b+1}z.\]

Next, suppose $i\leqslant b-1$. If $j_d\leqslant j-4$ we already have an expression in reduced form and the commuting relations alone put it into our normal form to give the stated result.

So let $j_d\geqslant j-2$. Suppose $i<b-1$ and $j=b+3$. Then we have
\begin{align*}
\Psichaind{j_d}{b+1}\cdot\Psichainu{i}{b-1}\Psi_{b+1}z&=\Psichaind{j_d}{b+3}\cdot\Psichainu{i}{b-3}\Psi_{b+1}\Psi_{b-1}\Psi_{b+1}z\\
&=\Psichaind{j_d}{b+1}\cdot\Psichainu{i}{b-3}z\\
&=0.
\end{align*}

Finally, let $j\geqslant b+5$. Then we have
\begin{align*}
\Psichaind{j_d}{b+1}\cdot\Psichainu{i}{b-1}\Psichaind{j-2}{b+1}z&=\Psichainu{i}{b-3}\Psichaind{j_d}{b-1}\Psichaind{j-2}{b+1}z\\
&=\Psichainu{i}{b-3}\Psichaind{j-6}{b-1}\Psichaind{j_d}{b+1}z
\end{align*}
which is reduced and in our normal form.
\end{proof}

\begin{prop}\label{triangular}
Suppose $v_T=\Psichaind{j_1}{b+3-2d}\Psichaind{j_2}{b+5-2d}\dots\Psichaind{j_d}{b+1}z\in\mathscr{D}$ is in reduced normal form, $i$ and $j$ are odd with $3\leqslant i\leqslant b+1<j\leqslant n$ and let \[(*)=\Psichaind{j_1}{b+3-2d}\Psichaind{j_2}{b+5-2d}\dots\Psichaind{j_d}{b+1}\cdot\Psichainu{i}{b-1}\Psichaind{j-2}{b+1}z.\] 

Then $(*)$ is a scalar multiple of either $v_T$ or some longer $\Psi$-expression. In particular, $(*)$ is a scalar multiple of $v_T$ in precisely the following cases:\\

$(*)=v_T$ if

\begin{itemize}
\item $i+j=2b+6$, $i\geqslant b+3-2d$, $j_v\geqslant j-4-4(d-v)$ for all $v$;
\item $i+j=2b+2$, $i\geqslant b+1-2d$ and $j_v \geqslant j-2-4(d-v)$ for all $v$.
\end{itemize}

$(*)=-2v_T$ if
\begin{itemize}
\item $i+j=2b+4$, $i\geqslant b+3-2d$ and $j_v\geqslant j-2-4(d-v)$ for all $v$.
\end{itemize}
\end{prop}

\begin{proof}
We will use the previous lemma and work down the cases in the order they appear above. We will always look to put expressions into reduced normal form.\\

\begin{enumerate}[leftmargin=*]
\item Let $d>0$. When $i=b+1$, we can clearly see that we get $(*)=-2v_T$ when $j=b+3$, $(*)=v_T$ when $j=b+5$ and $(*)=0$ otherwise. It is also clear that when $i=b-1$ and $j=b+3$ we have $(*)=v_T$, so in all further cases we will ignore this combination.\\

\item If $i\leqslant b-1$ and $j_d\leqslant j-4$, we must split into two subcases.

\begin{enumerate}
\item First suppose $i\leqslant b+1-2d$. Then we have
\begin{align*}
\Psichaind{j_1}{b+3-2d}\Psichaind{j_2}{b+5-2d}\dots\Psichaind{j_d}{b+1}\cdot\Psichainu{i}{b-1}\Psichaind{j-2}{b+1}z&=\Psichaind{j_1}{b+3-2d}\Psichaind{j_2}{b+5-2d}\dots\Psichaind{j_{d-1}}{b-1}\cdot\Psichainu{i}{b-3}\Psichaind{j_d}{b-1}\Psichaind{j-2}{b+1}z\\
&=\Psichainu{i}{b-1-2d}\Psichaind{j_1}{b+1-2d}\Psichaind{j_2}{b+3-2d}\dots\Psichaind{j_d}{b-1}\Psichaind{j-2}{b+1} z.
\end{align*}

The above expression is reduced and longer than $v_T$.\\

\item If $i\geqslant b+3-2d$, say $i=k_s-2=b-1-2d+2s$ for some $s\geqslant 2$, we have

\[(*)=\underbrace{\Psichaind{j_1}{b+3-2d}\dots\Psichaind{j_{s-1}}{b-1-2(d-s)}}_{=:\Psi^*}\cdot\Psichaind{j_s}{b-1-2(d-s)}\Psichaind{j_{s+1}}{b+1-2(d-s)} \dots\Psichaind{j_d}{b-1}\Psichaind{j-2}{b+1}z.\]

\textbf{\emph{Claim}} Suppose for some $s-1\leqslant u\leqslant d-1$ we have $j_v\geqslant b+3-2(d+s-2v)$ for all $s-1\leqslant v\leqslant u$. Then the above expression is equal to \[\Psi^* \ \Psichaind{j_s}{b+1-2(d-s)}\dots\Psichaind{j_u}{b+1+2(d-u)}\Psichaind{j_{u+1}}{b+7-2(d+s-2u)}\Psichaind{j_{u+2}}{b+3-2(d-u)}\dots\Psichaind{j_d}{b-1}\Psichaind{j-2}{b+1}z.\] If for the maximal such $u$ we have $u\leqslant d-2$, the expression above is reduced and longer than $v_T$.\\

Assuming the claim to be true, we need to look at what happens if the condition in the claim holds for $u=d-1$. In this instance, by the claim we have
\begin{align*}
(*)&=\Psi^* \ \Psichaind{j_s}{b+1-2(d-s)}\dots\Psichaind{j_{d-1}}{b-1}\Psichaind{j_d}{b+3+2(d-s)}\Psichaind{j-2}{b+1}z\\
&=\begin{cases}
-2v_T & \text{if }j=b+5+2(d-s)\text{ and }j_d\geqslant b+3+2(d-s),\\
v_T & \text{if }j=b+7+2(d-s)\text{ and }j_d\geqslant b+3+2(d-s),\\
0 & \text{if }j\geqslant b+9+2(d-s)\text{ and }j_d\geqslant b+3+2(d-s),\\
v_S & \text{otherwise, where $v_S$ is some expression longer than $v_T$.}\\
\end{cases}
\end{align*}

Note that the first case above never actually occurs here, by the condition that $j_d\leqslant j-4$. We can see that we get $(*)=v_T$ precisely when $j_v\geqslant b+3-2(d+s-2v)$ for all $s-1\leqslant v\leqslant d-1$ and $j_d=b+3+2(d-s)=j-4$.\\

Finally, we prove the claim, by induction on $u$. When $u=s-1$, we have that $j_{s-1}\geqslant b-1-2(d-s)$ (which we already knew a priori) and $j_s=b+1-2(d-s)$. Then \[\Psi^* \ \cdot\Psichaind{b+1-2(d-s)}{b-1-2(d-s)}\Psichaind{j_{s+1}}{b+1-2(d-s)}\dots\Psichaind{j_d}{b-1}\Psichaind{j-2}{b+1}z=\Psi^* \ \cdot \Psichaind{j_{s+1}}{b+1-2(d-s)}\dots\Psichaind{j_d}{b-1}\Psichaind{j-2}{b+1}z\] and the claim holds.

Suppose the claim is true for some $s-1\leqslant u\leqslant d-2$, and that $j_v\geqslant b+3-2(d+s-2v)$ for all $s-1\leqslant v\leqslant u+1$. Then by induction, we have
\begin{align*}
&\Psi^* \ \cdot\Psichaind{b+1-2(d-s)}{b-1-2(d-s)}\Psichaind{j_{s+1}}{b+1-2(d-s)}\dots\Psichaind{j_d}{b-1}\Psichaind{j-2}{b+1}z\\
&=\Psi^* \ \underbrace{\Psichaind{j_s}{b+1-2(d-s)}\dots\Psichaind{j_u}{b+1+2(d-u)}}_{=:\Psi^\dagger}\Psichaind{j_{u+1}}{b+7-2(d+s-2u)}\Psichaind{j_{u+2}}{b+3-2(d-u)}\dots\Psichaind{j_d}{b-1}\Psichaind{j-2}{b+1}z\\
&=\Psi^*\Psi^\dagger \ \Psichaind{j_{u+1}}{b+7-2(d+s-2u)}\Psichaind{j_{u+2}}{b+11-2(d+s-2u)}\Psichaind{b+5-2(d+s-2u)}{b+3-2(d-u)}\Psichaind{j_{u+3}}{b+5-2(d-u)}\dots\Psichaind{j_d}{b-1}\Psichaind{j-2}{b+1}z\\
&\\
& \qquad \text{since $j_{u+1}\geqslant b+7-2(d+s-2u)$ by hypothesis}\\
&\\
&=\Psi^*\Psi^\dagger \ \Psichaind{j_{u+1}}{b+3-2(d-u)}\Psichaind{j_{u+2}}{b+11-2(d+s-2u)}\Psichaind{j_{u+3}}{b+5-2(d-u)}\dots\Psichaind{j_d}{b-1}\Psichaind{j-2}{b+1}z
\end{align*}

and the claim is proved.

\end{enumerate}
\item Next, we look at the final case, $i\leqslant b-1$, $j_d\geqslant j-2$ and $j\geqslant b+5$. We have that \[(*)=\Psichaind{j_1}{b+3-2d}\Psichaind{j_2}{b+5-2d}\dots\Psichaind{j_{d-1}}{b-1}\cdot\Psichainu{i}{b-3}\Psichaind{j-6}{b-1}\Psichaind{j_d}{b+1}z\]

Once again, we split into subcases.

\begin{enumerate}
\item First suppose $i\leqslant b+1-2d$. Then \[(*)=\Psichainu{i}{b-1-2d}\Psichaind{j_1}{b+1-2d}\Psichaind{j_2}{b+3-2d}\dots\Psichaind{j_{d-1}}{b-3}\Psichaind{j-6}{b-1}\Psichaind{j_d}{b+1}z.\]

\textbf{\emph{Claim}} Suppose for some $0\leqslant u\leqslant d-1$ we have  $j<j_{d-v}+4(v+1)$ for all $0\leqslant v\leqslant u$ and $j\geqslant b+3+2v$ for all $0\leqslant v\leqslant u$. Then \[(*)=\Psichainu{i}{b-1-2d}\Psichaind{j_1}{b+1-2d}\Psichaind{j_2}{b+3-2d}\dots\Psichaind{j_{d-u-1}}{b-3-2u}\Psichaind{j-6-4u}{b-1-2u}\Psichaind{j_{d-u}}{b+1-2u}\dots\Psichaind{j_d}{b+1}z\] and this expression is of length $2(u+1)$ less than the length of $(*)$. Furthermore, if we take the maximal such $u$ and have $u\leqslant d-2$ and $j\geqslant b+5+2u$, it is reduced. If $j=b+3+2u$, the expression is equal to 0.\\

We prove the claim by induction on $u$. If $u=0$, the result follows immediately from Lemma \ref{triangularsimple}. Now suppose the claim holds for some $0\leqslant u\leqslant d-2$, and that $j<j_{d-v}+4(v+1)$ for all $0\leqslant v\leqslant u+1$, but $j\geqslant j_{d-u-2}+4(u+3)$ (if $u\leqslant d-3$). Then by induction, we have

\begin{align*}
(*)&=\underbrace{\Psichainu{i}{b-1-2d}\Psichaind{j_1}{b+1-2d}\Psichaind{j_2}{b+3-2d}\dots\Psichaind{j_{d-u-2}}{b-5-2u}}_{=:\Psi^*}\Psichaind{j_{d-u-1}}{b-3-2u}\Psichaind{j-6-4u}{b-1-2u}\Psichaind{j_{d-u}}{b+1-2u}\dots\Psichaind{j_d}{b+1}z\\
&=\Psi^*\Psichaind{j_{d-u-1}}{j-4-4u}\Psi_{j-6-4u}\Psi_{j-8-4u}\Psi_{j-6-4u}\Psichaind{j-10-4u}{b-3-2u}\Psichaind{j-8-4u}{b-1-2u}\Psichaind{j_{d-u}}{b+1-2u}\dots\Psichaind{j_d}{b+1}z\\
&=\Psi^*\Psichaind{j-10-4u}{b-3-2u}\Psichaind{j_{d-u-1}}{b-1-2u}\Psichaind{j_{d-u}}{b+1-2u}\dots\Psichaind{j_d}{b+1}z,
\end{align*}

which is the claimed expression. In the induction step, 2 $\Psi$ terms have been deleted, which proves the length part of the claim. It is clear that if $j<b+5+2u$ then the expression in the claim is 0 and likewise that when $u\leqslant d-2$ (and $j\geqslant b+5+2u$), we have a reduced expression.

Now, let $u$ be maximal under the conditions in the claim. First, suppose that $u\leqslant d-2$. By the claim, we can assume that $j\geqslant b+5+2u$. This implies that $\Psichaind{j-2}{b+1}$ has length at least $u+2$. Similarly, $i\leqslant b+1-2d$ and $u\leqslant d-2$ imply $\Psichainu{i}{b-1}$ also has length at least $u+2$. So by the claim, once $(*)$ has been written in a reduced form, it has length at least 2 more than $v_T$.

But what if $u=d-1$? The above claim tells us that \[(*)=\Psichainu{i}{b-1-2d}\Psichaind{j-2-4d}{b+1-2d}v_T.\] This is zero  unless $j\geqslant b+3+2d$, in which case we have a (reduced) longer expression than $v_T$, or $i\geqslant b+1-2d$. Note that in the latter case, we in fact have $i=b+1-2d$ because of the conditions on the subcase we are looking at. Looking at this case, we assume $j<b+3+2d$, since $j\geqslant b+3+2d$ yields an expression longer than $v_T$. Under these conditions, we have $(*)=v_T$.

\item Finally, suppose that $i\geqslant b+3-2d$. Say $i=k_s-2=b-1-2(d-s)$ for some $s\geqslant 2$. Then \[(*)=\Psichaind{j_1}{b+3-2d}\dots\Psichaind{j_{s-1}}{b-1-2(d-s)}\cdot\Psichaind{j_s}{b-1-2(d-s)}\Psichaind{j_{s+1}}{b+1-2(d-s)}\dots\Psichaind{j_d}{b-1}\Psichaind{j-2}{b+1}z.\]

\textbf{\emph{Claim 1}} Suppose we have $-1\leqslant u\leqslant d-s-1$ with $j_{s+v}\geqslant b+3-2(d-s)+4v$ for all $-1\leqslant v\leqslant u$. Let \[\Psi^*:=\Psichaind{j_1}{b+3-2d}\dots\Psichaind{j_{s+u}}{b+1-2(d-(s+u))}.\] Then \[(*)=\Psi^*\Psichaind{j_{s+u+1}}{b+7-2(d-s-2u)}\cdot\Psichaind{j_{s+u+2}}{b+3-2(d-(s+u))}\Psichaind{j_{s+u+3}}{b+5-2(d-(s+u))}\dots\Psichaind{j_d}{b-1}\Psichaind{j-2}{b+1}z.\] The above claim is proved by a simple but tedious induction, in the spirit of previous claims in this proof. Now first suppose we have $u=d-s-1$ satisfying the conditions in the claim, but also $j_d\geqslant b+3+2(d-s)$. Then

\begin{align*}
(*) &= \Psichaind{j_1}{b+3-2d}\dots\Psichaind{j_{d-1}}{b-1}\Psichaind{j_d}{b+3+2(d-s)} \Psichaind{j-2}{b+1} z\\
&= \begin{cases}
-2T' & \text{if $j=b+5+2(d-s)$,}\\
T' & \text{if $j=b+3+2(d-s)$ or $j=b+7+2(d-s)$,}\\
0 & \text{otherwise.}\\
\end{cases}
\end{align*}

Otherwise, take $u$ maximal, satisfying the conditions of claim 1. We have \[(*)=\Psi^*\cdot\Psichaind{j_{s+u+2}}{b+3-2(d-(s+u))}\Psichaind{j_{s+u+3}}{b+5-2(d-(s+u))}\dots\Psichaind{j_d}{b-1}\Psichaind{j-2}{b+1}z,\] by the claim. Note that for these conditions on $u$ to hold, we have $j_{s+u}=b+3-2(d-s)+4u$ and $j_{s+u+1}=b+5-2(d-s)+4u$.

\textbf{\emph{Claim 2}} Let $0\leqslant r\leqslant d-s-u-2$ be such that $j_{d-v}\geqslant j-2-4v$ for all $0\leqslant v\leqslant r$. Then \[(*)=\Psi^*\Psichaind{j_{s+u+2}}{b+3-2(d-(s+u))}\dots\Psichaind{j_{d-r-1}}{b-3-2r}\Psichaind{j-6-4r}{b-1-2r}\Psichaind{j_{d-r}}{b+1-2r}\Psichaind{j_{d-r+1}}{b+3-2r}\dots\Psichaind{j_d}{b+1}z.\] If r is maximal such that $j_{d-v}\geqslant j-2-4v$ for all $0\leqslant v\leqslant r$ and $r\leq d-s-u-3$, then this expression is reduced.

Again, this claim can be proved by induction as with the previous claims. Note that the above term is zero unless  $j\geqslant b+5+2r$.\\

Whenever $r\leqslant d-s-u-3$, the reduced expression above is longer than $v_T$. To see why, note that we have the condition $j_{s+u+1}=b+5-2(d-s-2u)$ from claim 1. Since $j_{i+1}\geqslant j_i+2$, this yields $j_{d-r-1}\geqslant b+1+2(u-r)$. Now, we have assumed that $j_{d-r-1}<j-6-4r$, so we can combine these inequalities to yield $j\geqslant b+9+2(u+r)$.

We now have enough information to compare lengths. To leave this reduced form, we first deleted $2u+4$ $\Psi$ terms from $(*)$ to arrive at the result from claim 1. Next, we deleted $2r+2$ $\Psi$ terms to arrive at the result of claim 2. So in total, we have deleted $2(r+u+3)=:\delta$ $\Psi$ terms from $(*)$ to leave a reduced expression.

Now, how many $\Psi$ terms did we append to $v_T$ in the definition of $(*)$? Call the number of terms appended $\alpha$. Since $i=b-1-2(d-s)$, $\Psi \hspace{-4pt} \underset{\scriptscriptstyle{i}}{\overset{\scriptscriptstyle{b-1}}{\uparrow}}$ is a product of $d-s+1$ $\Psi$ terms. Since $j\geqslant b+9+2(u+r)$, $\Psichaind{j-2}{b+1}$ has length at least $4+u+r$. So, $\alpha\geqslant d+u+r-s+5$. By the definition of $r$, we have that $d-s\geqslant r+u+2$, so $\alpha\geqslant 2(u+r+3)+1>\delta$, and we are done.

Now suppose $r=d-s-u-2$ satisfies the conditions of claim 2, and we are left with a reduced expression. The claim tells us that the reduced expression is \[(*)=\Psi^*\Psichaind{j+2-4(d-(s+u))}{b+3-2(d-(s+u))}\Psichaind{j_{s+u+2}}{b+5-2(d-(s+u))}\dots\Psichaind{j_d}{b+1} z.\] Now, using the fact that $j_{s+u}=b+3-2(d-s)+4u$, we see that for this expression to be reduced we have $j\geqslant b+3+2(d-s)$. Now, arguing as above, we have $\delta=2(r+u+3)=2(d-s+1)$, the length of $\Psichainu{i}{b-1}$ is once again $d-s+1$ and the length of $\Psichaind{j-2}{b+1}$ is at least $d-s+1$. Hence $\alpha\geqslant\delta$, with equality precisely when $j=b+3+2(d-s)$, in which case we have $(*)=v_T$.

Now suppose $r=d-s-u-2$ satisfies the conditions of claim 2, but we are not left with a reduced expression. Then \[(*)=\Psi^* \Psichaind{j+2-4(d-(s+u))}{b+3-2(d-(s+u))}\underbrace{\Psichaind{j_{s+u+2}}{b+5-2(d-(s+u))}\dots\Psichaind{j_d}{b+1}}_{=:\Psi^{**}}z\] which is zero unless $j\geqslant b+1+2(d-(s+u))$, and we have the following:

\textbf{\emph{Claim 3}} Let $-1\leqslant x\leqslant s+u-1$ be such that $j_{s+u-v}\geqslant j+2-4(d+v-(s+u))$ for all $-1\leqslant v\leqslant x$. Then \[(*)=\Psichaind{j_1}{b+3-2d}\dots\Psichaind{j_{s+u-x-1}}{b-1-2(d+x-(s+u))}\Psichaind{j-2-4(d+x-(s+u))}{b+1-2(d+x-(s+u))}\Psichaind{j_{s+u-x}}{b+3-2(d+x-(s+u))}\dots\Psichaind{j_{s+u}}{b+3-2(d-(s+u))}\Psi^{**}z.\] Note that if $x\leqslant s+u-2$, this term is zero unless $j\geqslant b+3+2(d+x-(s+u))$.

Take $x$ to be the maximal such that the conditions in claim 3 are met. First, suppose $x\leqslant s+u-2$. Then $j_{s+u-x}\geqslant j+2-4(d+x-(s+u))$ and $j_{s+u-x-1}<j-2-4(d+x-(s+u))$. When $x\leqslant u-1$, we have our assumption in using claim 1 that $j_{s+u-x-1}\geqslant b-1-2(d-s)+4(u-x)$. Comparing these inequalities yields $j\geqslant b+3+2(d-s)$.

Similarly if $x\geqslant u$, $j_{s+u-x-1}\geqslant b-1-2(d+x-(s+u))$ can be read off from the expression in claim 3. But this yields $j\geqslant b+3+2((d-s)+(x-u))\geqslant b+3+2(d-s)$. Now in either case,

\begin{align*}
b+3-2(d-s)+4u&=j_{s+u}\text{ by the comment after claim 1,}\\
&\geqslant j+2-4(d-(s+u))\text{ by the conditions in claim 3,}\\
&\geqslant b+5-2(d-s)+4u.
\end{align*}

We have a contradiction, and so if $j_{s+u}\geqslant j+2-4(d-(s+u))$ but $j_1<j+6-4d$, we must have $(*)=0$.

Now suppose $x=s+u-1$. Then we have $j\geqslant b-1+2d$, or else $(*)=0$ and we're done. So

\begin{align*}
b+3-2d+2s+4u&=j_{s+u}\text{ by the comment after claim 1,}\\
&\geqslant j+2-4(d-(s+u))\text{ by the conditions in claim 3,}\\
&\geqslant b+1-2d+4(s+u).
\end{align*}

This implies $s=1$, and so $i=b+1-2d$. But this breaks the initial conditions of the subcase we are in, so we again have a contradiction. So in fact we never get terms that look like the expression in claim 3; $\Psi^*$ remains intact in the final reduced expression for $(*)$, if it is non-zero.
\end{enumerate}
\end{enumerate}

If we collect the cases where $(*)$ is equal to a scalar multiple of $v_T$, we get the following list:

$(*)=v_T$ if
\begin{itemize}
\item $i=b+1$, $j=b+5$, $d>0$ -- from case 1;
\item $b+3-2d\leqslant i\leqslant b-1$, $j=2b+6-i$, $j_d=2b+2-i$, and $j_v\geqslant j-4-4(d-v)$ for all $v$ -- from case 2(b);
\item $b+3-2d\leqslant i\leqslant b-1$, $j=2b+6-i$ and $j_v\geqslant j-4-4(d-v)$ for all $v$ -- from case 3(b);\\
\item $i=b-1$, $j=b+3$, $d>0$ -- from case 1;
\item $i=b+1-2d$, $j=b+1+2d\geqslant b+5$ and $j_v\geqslant j-2-4(d-v)$ for all $v$ -- from case 3(a);
\item $b+3-2d\leqslant i$, $j=2b+2-i\geqslant b+5$, $j_d\geqslant j-2$ and $j_v\geqslant j-2-4(d-v)$ for all $v$ -- from case 3(b);
\item $b+3-2d\leqslant i$, $j=2b+2-i\geqslant b+5$, $j_d\geqslant j-2$ and $j_v\geqslant j-4(d-v)$ for all $v$ -- from case 3(b).
\end{itemize}

These conditions can be written compactly as the first and second conditions in the statement of the proposition.

$(*)=-2v_T$ if
\begin{itemize}
\item $i=b+1$, $j=b+3$, $d>0$ -- from case 1;
\item $b+3-2d\leqslant i\leqslant b-1$, $j=2b+4-i\geqslant b+5$, and $j_v\geqslant j-2-4(d-v)$ for all $v$ -- from case 3(b).
\end{itemize}

These two conditions can be written compactly as the final condition in the statement of the proposition.
\end{proof}

The above result immediately leads to the following crucial fact.

\begin{cor}
Order $\mathscr{D}$ so that $v_U$ comes after $v_T$ whenever $r(U)>r(T)$. With respect to this ordering, the action of $f$ on $e(i_\lambda)S_\lambda$ is lower triangular. In particular, for each $T\in\Dom(\lambda)$, the coefficient of $v_T$ in $f(v_T)$ is an eigenvalue of $f$.
\end{cor}

\begin{prop}\label{eigenvalues}
$f$ has the eigenvalues \[-\mfrac{d}{2}(n-2d+1)\quad\text{for } \ d=0,1,\dots,b/2.\]
\end{prop}

\begin{proof}
Fix some $d\in\{0,1,\dots,b/2\}$. Let \[v_T=\Psichaind{n-2d}{b+3-2d}\dots\Psichaind{n-4}{b-1}\Psichaind{n-2}{b+1}z.\] Using the three bullet points in Proposition \ref{triangular}, we will compute the eigenvalue $-\mfrac{d}{2}(n-2d+1)$ as the coefficient of $v_T$ in $f(v_T)$. First, note that by choice of $T$ the inequality on $j_d$ for each bullet point is the strongest. So to check when the family of inequalities at the end of each bullet point holds, it suffices to only verify the inequality on $j_d$.

If $i+j=2b+6$ and $i\geqslant b+3-2d$, then we claim that the inequalities in the first bullet point are always satisfied by $v_T$. For this, we need $j_d\geqslant 2b+2-i$. Using that $d\leqslant b/2$ and $n>2b$ we also get that $2b+2-i\leqslant b-1+2d\leqslant 2b-1\leqslant n-2$. So the inequalities always hold in the case of the first bullet point.

Now, if $i+j=2b+2$ and $i\geqslant b+1-2d$, we claim that the inequalities in the second bullet point are always satisfied by $v_T$. To see this, we must show that $j_d\geqslant 2b-i$. We have $2b-i\leqslant b-1+2d\leqslant 2b-1\leqslant n-2$ and so the inequalities always hold in the case of the second bullet point.

Finally, $i+j=2b+4$ and $i\geqslant b+3-2d$, then we claim that the inequalities in the third bullet point are always satisfied by $v_T$. We need $j_d\geqslant 2b+2-i$ but have $2b+2-i\leqslant b-1+2d\leqslant 2b-1\leqslant n-2$ and we are done.

So now we only need to verify which pairs $(i,j)$ satisfy the first two conditions in each bullet point. For the first bullet point, we have the pairs $(b+3-2d,b+3+2d),(b+5-2d,b+1+2d),\dots,(b+1,b+5)$. For the second, we have the pairs $(b+1-2d,b+1+2d),(b+3-2d,b-1-2d),\dots,(b-1,b+3)$. For the third, we have the pairs $(b+3-2d,b+1+2d),(b+5-2d,b-1+2d),\dots,(b+1,b+3)$.

Recall that the coefficient of $\Psichainu{i}{b-1}\Psichaind{j-2}{b+1}$ in $f(z)$ is $\mfrac{i-1}{2}\cdot\mfrac{n+2-j}{2}$. Hence the coefficient of $v_T$ in $f(v_T)$ is

\begin{align*}
&\mfrac{1}{4}\sum_{r=0}^{d-1}(b-2r)(a-3-2r)+(b-2-2r)(a-1-2r)-2(b-2r)(a-1-2r)\\
&=\mfrac{1}{4}\sum_{r=0}^{d-1}8r-2(n-1)\\
&=-\mfrac{1}{2}d(n-1)+2\sum_{r=0}^{d-1}r\\
&=-\mfrac{d}{2}(n-2d+1).\qedhere
\end{align*}
\end{proof}

\begin{rem}
The sequence of eigenvalues given above is \[0,-\mfrac{(n-1)}{2},-(n-3),-\mfrac{3}{2}(n-5),\dots,-\mfrac{b}{4}(a+1).\] If we write $a=2r+1$ and $b=2s$ then this sequence can be rewritten as \[0,-(r+s),-2(r+s-1),-3(r+s-2),\dots,-s(r+1).\]
\end{rem}

\begin{thm}
Suppose $\Char\mathbb{F}\neq 2$. Then $S_{(a,1^b)}$ is decomposable if either $b\geqslant 4$ or $b=2$ with $\Char\mathbb{F}\nmid\mfrac{n-1}{2}$.
\end{thm}

\begin{proof}
By the remark after Proposition \ref{endomorphismf}, it suffices to show that $f$ has at least two distinct eigenvalues. When $s\geqslant 2$, $0$, $-\mfrac{(n-1)}{2}$ and $-(n-3)$ are three eigenvalues of $f$; if $S_{(a,1^b)}$ were indecomposable, these would be equal. Since $p\neq 2$, this is impossible, and we have the desired result.

When $b=2$ and $\Char\mathbb{F}\nmid \mfrac{n-1}{2}$, we have the distinct eigenvalues $0$ and $-\mfrac{(n-1)}{2}$ and we are done.
\end{proof}

It remains to resolve the case $a=2r+1$, $b=2$ when $\Char\mathbb{F}\mid\mfrac{n-1}{2}$. We have \[f(z)=r\cdot v_{T_{3,b+3}}+(r-1)\cdot v_{T_{3, b+5}}+\dots+v_{T_{3,n}}=\sum_{c=1}^rc\cdot\Psichaind{3+2(r-c)}{3}z.\]

When $b=2$ and $\Char\mathbb{F}\mid\mfrac{n-1}{2}$, we will prove that $S_\lambda$ is indecomposable by showing that $\End_\mathscr{H}(S_\lambda)$ has no non-trivial idempotents.

\begin{lem}
Suppose $a$ is odd. Then $\{I,f\}$ is a basis of $\End_\mathscr{H}(S_{(a,1^2)})$, where $I$ is the identity map on $S_{(a,1^2)}$.
\end{lem}

\begin{proof}
Suppose we have $g\in S_{(a,1^2)}\setminus\langle I,f\rangle_\mathbb{F}$. Since the coefficient of $v_{T_{3,n}}$ in $f$ is 1, we can add multiples of $I$ and $f$ to assume without loss of generality that \[g(z)=\sum_{j=2}^{(n-3)/2}\alpha_jv_{T_{3,2j+1}}=\sum_{j=2}^{(n-3)/2}\alpha_j\Psichaind{2j-1}{3}z.\]

We will show that applying the relations $\psi_{n-2k}g(z)=0$ for $k=1,2,\dots,(n-5)/2$ yields $\alpha_{(n-2k-1)/2}=0$. It then follows that $g$ is the zero map, a contradiction.

Suppose, by induction on $k$, we have \[g(z)=\sum_{j=2}^{(n-2k-1)/2}\alpha_j\Psichaind{2j-1}{3}z.\] Then, acting on $g(z)$ by $\psi_{n-2k}$ yields $\alpha_{(n-2k-1)/2}\psi_{n-2k}\Psichaind{n-2k-2}{3}z=0$ and we are done.
\end{proof}

In order to find idempotents, we would like to know how to compose elements of our basis. This amounts to the following lemma.

\begin{lem}\label{f2}
Let $a=2r+1$ and $b=2$. Then $f^2(z)=-(r+1)f(z)$.
\end{lem}

\begin{proof}
Notice that $\Psi_3\cdot\Psichaind{3+2(r-c)}{3}z=0$ for all $c\leqslant r-2$. So

\begin{align*}
f^2(z)&=\sum_{c=1}^rc\cdot\Psichaind{3+2(r-c)}{3}f(z)\\
&=\sum_{c=1}^rc\cdot\Psichaind{3+2(r-c)}{3}(r-1\Psi_5\Psi_3z+r\Psi_3z)\\
&=\sum_{c=1}^rc\cdot\Psichaind{3+2(r-c)}{3}(-(r+1)z)\\
&=-(r+1)f(z).\qedhere
\end{align*}
\end{proof}

\begin{lem}
Suppose $a=2r+1$ and $\Char\mathbb{F}\mid\mfrac{n-1}{2}$. Then the only idempotents in $\End_\mathscr{H}(S_{(a,1^2)})$ are $0$ and $I$, and hence $S_{(a,1^2)}$ is indecomposable.
\end{lem}

\begin{proof}
Let $\alpha,\beta\in\mathbb{F}$. Using Lemma \ref{f2}, we have $f^2(z)=0$ and therefore \[(\alpha I+\beta f)^2=\alpha^2I+2\alpha\beta f.\]

So $\alpha I+\beta f$ is an idempotent if and only $\alpha^2=\alpha$ and $2\alpha\beta=\beta$.

Whether $\alpha=0$ or $\alpha=1$, we must have $\beta=0$. The result follows.
\end{proof}

With the aid of Murphy's result in \cite{gm}, we have now completely determined decomposability of the Specht modules $S_{(a,1^b)}$. We summarise our result in the following theorem.

\begin{thm}
Suppose $\Char\mathbb{F}\neq 2$. Then $S_{(a,1^b)}$ is indecomposable if and only if $n$ is even, or $b=2\text{ or }3$ and $\Char\mathbb{F}|\lceil \frac{a}{2}\rceil$.
\end{thm}

\end{document}